\documentclass[a4paper, 11pt]{amsart}
\usepackage{graphicx}
    
\usepackage{dsfont}
\usepackage{ stmaryrd }
\usepackage{lineno} 
\usepackage[colorlinks]{hyperref}
\usepackage[nameinlink,noabbrev]{cleveref}
\usepackage{enumitem}
\usepackage{lineno} 
\usepackage{multicol}
\usepackage[left=2cm,top=2cm,right=2cm,bottom=2cm]{geometry}
\usepackage{amsthm,amssymb,mathrsfs,pstricks,booktabs}
\usepackage{mathtools,amsmath,fancyhdr,appendix,nccmath}
\usepackage[all]{xy} 
\usepackage{enumitem}
\usepackage{array} 
\usepackage{xcolor}
 \usepackage{verbatim} %para poder comentar 
 
\definecolor{ao}{rgb}{0.0, 0.5, 0.0}

\hypersetup{
    linkcolor=ao,
    urlcolor=ao,
    citecolor=ao,
    hyperindex=true,
    bookmarks=true,
    pagebackref=true
}

\usepackage{dsfont}

\title{Formal superschemes over fields: Basic theory}
\author{Felipe Saenz}
\address{Felipe Saenz \\Instituto de Matemática Pura e Aplicada\\ Rio de Janeiro RJ\\ Brazil }
\email{felipe.leon@impa.br}

\author{Joel Torres del Valle}
\address{Joel Torres del Valle\\Instituto de Matemáticas\\ FCEyN\\Universidad de Antioquia\\50010 Medellín\\ Colombia}
\email{joel.torres@udea.edu.co}
\date{\today}
\theoremstyle{plain}
\newtheorem{theorem}{Theorem}[section]
\newtheorem{lemma}[theorem]{Lemma}
\newtheorem{corollary}[theorem]{Corollary}
\newtheorem{proposition}[theorem]{Proposition}

\theoremstyle{definition}
\newtheorem{definition}[theorem]{Definition}
\newtheorem{example}[theorem]{Example}
\newtheorem{question}[theorem]{Question}

\newtheorem{remark}[theorem]{Remark}

\makeatletter
\def\proof{\par\pushQED{\qed}\normalfont\topsep6\p@ \trivlist \item[\hskip 20pt \itshape\proofname\@addpunct{.}\hskip\labelsep] \ignorespaces}
\makeatother

\newcommand\function[5]{%
  \begingroup
  \setlength\arraycolsep{0pt}
  #1\colon\begin{array}[t]{c >{{}}c<{{}} c}
             #2 & \to & #3 \\ #4 & \mapsto & #5 
          \end{array}%
  \endgroup}

\newcommand{\Ub}{\mathbb{U}}
\newcommand{\Us}{\mathsf{U}}
\newcommand{\Vs}{\mathsf{V}}
\newcommand{\Xs}{\mathsf{X}}
\newcommand{\Ys}{\mathsf{Y}}
\newcommand{\Ws}{\mathsf{W}}
\newcommand{\M}{\mathbf{M}}
\newcommand{\N}{\mathbf{N}}
\newcommand{\A}{\mathbf{A}}
\newcommand{\B}{\mathbf{B}}
\newcommand{\f}{\mathbf{f}}

\newcommand{\p}{\mathfrak{p}}
\renewcommand{\a}{\mathfrak{a}}
\newcommand{\m}{\mathfrak{m}}
\newcommand{\spec}{\,\mathrm{Spec}}

\newcommand{\Z}{\mathbb{Z}}
\newcommand{\cod}{\mathrm{corad}}
\newcommand{\ev}{_{\overline{0}}}
\newcommand{\od}{_{\overline{1}}}

\renewcommand{\qed}{\hfill$\square$}
\renewcommand{\to}
{\longrightarrow}

\renewcommand{\mapsto}{\longmapsto}

\renewcommand{\O}{\mathcal{O}}
\renewcommand{\hom}{\mathrm{Hom}}
 
\newcommand{\homm}{\underline{\mathrm{Hom}}}

\newcommand{\kdim}{\mathrm{Kdim}}
\newcommand{\ksdim}{\mathrm{Ksdim}}
\newcommand{\sdim}{\mathrm{sdim}}

\newcommand{\K}{\Bbbk}
\newcommand{\alg}{\mathsf{SAlg}_\K}
\newcommand{ \sets}{\mathsf{Sets}}
 
\newcommand{\SF}{\mathsf{SFunt}_\K}
\renewcommand{\ss}{\mathsf{SSch}_\K}
\newcommand{\bX}{\mathfrak{X}}
\newcommand{\bY}{\mathfrak{Y}}

\newcommand{\bW}{\mathfrak{W}}
\renewcommand{\sp}{\mathrm{SSp}}
\newcommand{\V}{\mathfrak{V}}
\newcommand{\D}{\mathfrak{D}}
\newcommand{\g}{\mathbf{g}}
\newcommand{\sspec}{\mathrm{SSpec}} %superespectro 

\newcommand{\falg}{\mathsf{SAlgf}_\K}
\newcommand{\SfF}{\mathsf{SFFunt}_\K}
\newcommand{\spf}{\mathrm{SSpf}}
\newcommand{\salg}{\mathsf{PSAlg}_\K}

\keywords{Formal superscheme, super-coalgebra, super $\K$-functor}
\subjclass[2020]{16T15, 14A15, 14A22}
\begin{document}

\begin{abstract}  
 This paper develops the basic theory of formal schemes over fields in the supersymmetric setting. We introduce the notion of a formal superscheme and investigate some of its fundamental properties. Particular emphasis is placed on the study of morphisms between formal superschemes, for which we establish a faithfully flat descent theorem and a fiber dimension-type theorem. 
\end{abstract}

\maketitle

\section{Introduction}

The theory of formal schemes is well established in the commutative setting, especially under the Noetherianity assumption (see e.g., \cite{fujiwara2018foundations}, \cite{grothendieck1960elements} and \cite{hartshorne2013algebraic}). Beyond the classical approach based on locally ringed spaces, a functorial perspective reformulates schemes in terms of $\K$-functors, and formal schemes as inductive limits of direct systems of finite affine $\K$-schemes (see e.g., \cite{demazure1980introduction}, \cite{demazure2006lectures} and \cite{takeuchi1977formal}). This perspective is particularly powerful in the study of linear algebraic groups, representation and deformation theory (see, e.g., the classic texts \cite{demazure2006lectures}, \cite{jantzen2003representations}).

In the realm of supergeometry, only partial developments of a theory of formal superschemes have been proposed to date. More precisely, the notion of $\K$-functor has been extended to the ``super'' setting in works such as~\cite{masuoka2011quotient} and \cite{zubkov2009affine}, where several generalizations of classical results on affine group superschemes are achieved. More recently, \cite{takahashi2024quotients} introduced a formalization of the concept of formal superschemes and translated some results on quotients of formal supergroups into this framework. A related contribution can be found in~\cite{moosavian2019existence} and \cite{zubkov2024automorphism}, where various aspects of the theory are developed from the standpoint of locally ringed superspaces, under Noetherianity assumption.

As the authors of \cite{takahashi2024quotients} note, a systematic, field-based theory of formal superschemes (similar to 
\cite{takeuchi1977formal}) is still lacking. The present work aims to  contribute to address this gap by initiating the construction of such a theory. We cover topics such as flat and faithfully flat morphisms, locally algebraic superschemes, and superdimension theory.

\subsection{Our contribution} 

{Hereinafter we consider

\begin{align*}
    \K&:\text{a field with characteristic not equal to 2}\\
    \alg&:\text{the category of supercommutative $\K$-superalgebras,}\\
    \falg&:\text{the full subcategory of $\alg$ consisting of all finite-dimensional  $\K$-superalgebras, }\\
    \sets&:\text{the category of sets,}\\
    \mathsf{SCoalg}_\K&:\text{the category of super-cocommutative super-coalgebras over $\K$,}\\
    \salg&:\text{the category of profinite $\K$-superalgebras.}
\end{align*} 

A \textit{super $\K$-functor} is a covariant functor from $\alg$ to $\sets$. A \textit{formal super $\K$-functor} is a covariant functor from $\falg$ to $\sets$.

If $A\in\falg$, we define the formal super $\K$-functor $\spf\,A$ by 

\[
\spf\,A:B\mapsto\hom_{\falg}(A, B),\quad \text{for all }B\in\falg.
\]

\begin{definition}   A \textit{formal $\K$-superscheme} $\bX$ is a formal super $\K$-functor isomorphic to the inductive limit of a filtered inductive system of \textit{finite affine} $\K$-superschemes $(\bX_\lambda, u_{\lambda\mu})$; that is, each $\bX_\lambda$ is of the form $\spf\, A_\lambda$ with $A_\lambda \in \falg$, and the transition morphisms $u_{\lambda\mu} : \bX_\lambda \to \bX_\mu$ for $\lambda < \mu$ are closed immersions. 
\end{definition}

Given $\A\in\mathsf{SCoalg}_\K$, we define the functor $\sp^*\A:\falg\to\sets$ by 

\[
\sp^*\A :R \mapsto G_R(R\otimes \A),\quad \text{for all }R \in \falg.
\]

Here $G_R(R\otimes\A)$ denotes the \textit{group-like} elements in $(R\otimes \A)\ev$, that is, elements $u\in(R\otimes \A)\ev$ such that $\Delta(u)=u\otimes u$ and $\epsilon(u)=1$, with $\Delta$ and $\epsilon$ the coproduct and counit of the $R$-super-coalgebra $R\otimes \A$, respectively.

Now, for each $A\in\salg$ and $B\in\falg$, we consider $\hom_{\salg}(A, B)$ to consist of all continuous maps of $\K$-superalgebras $A\to B$, where we endow $B$ with the discrete topology. Then, we may define  $\mathrm{PSSpf}\,A$ as

\[
\mathrm{PSSpf}\,A:B\mapsto\hom_{\salg}(A, B), \quad\text{ for all }B\in\falg. 
\]

The results in the third section of this paper can be summarized as follows:

\begin{theorem}\label{the.1.2.intr}
    \ 

    \begin{itemize}
    \item[\rm i)] The assignment $A\mapsto\mathrm{PSSpf}\,A$ is an anti-equivalence of the category $\salg$ with the category of formal $\K$-superschemes. 
    \item[\rm ii)]  Let $A\in\falg$ and let $A^*\in\mathsf{SCoalg}_\K$ be its dual. The functors $\spf\,A$ and $\sp^*\, A^*$ are isomorphic.
   \item[\rm iii)] The assignment $\A\mapsto \sp^* \A$ is a covariant functor from $\mathsf{SCoalg}_\K$ to the category of $\K$-superfuntors. In addition, it induces an equivalence between  $\mathsf{SCoalg}_\K$ and the category of $\K$-formal superschemes.  
   \item[\rm iv)]  A formal $\K$-functor is a formal $\K$-superscheme if it commutes with finite projective limits. 
\end{itemize} 
\end{theorem}

It is worth noting that our notion of formal superscheme differs slightly from the one considered in~\cite{takahashi2024quotients}. In their approach, the super $\K$-functor $\sp^* \A$ is defined over the category $\alg$, whereas we restrict ourselves to the full subcategory $\falg$. This distinction does not hinder the applicability of Takeuchi's arguments \cite{takeuchi1977formal}, but it proves useful when ``translating" certain techniques from \cite{demazure2006lectures} into the language of supergeometry (see \Cref{sect.3.4}).

As previously mentioned, \cite{zubkov2024automorphism} extended the notion of a formal scheme in the sense of Hartshorne \cite{hartshorne2013algebraic} to the supergeometric setting. A natural question arises as to how this approach relates to the one adopted in the present work. Let $\bX=\sspec\,A$ be a Noetherian affine superscheme, and let $\bY=\sspec\,A/\a \subseteq \bX$ be a closed subsuperscheme given by a sheaf of superideals $\mathfrak{I} \subseteq \O_\bX$ corresponding to the superideal $\a$ of $A$. Equipping $A$ with the $\a$-adic topology, the formal completion of $\bX$ along $\bY $ as given in \cite[\S4.1]{zubkov2024automorphism} gives 

$$\mathrm{PSSpf}(A)(R) \simeq\varinjlim_n\,\mathrm{PSSpf}(A/\a^n)(R),\quad R\in\falg$$ 
in agreement with \Cref{the.1.2.intr} i).}   Furthermore, fiber products behave well, and flat morphisms are defined analogously as in the case of formal schemes and retain similar properties (see \Cref{sec.morphisms}). In particular, we supersize the theorem of faithfully flat descent (\cite[Theorem 2.4]{takeuchi1977formal}):

\begin{theorem} If $\f:\bX\rightarrow\bY$ is a faithful flat map of formal superschemes, then the diagram \[
\xymatrix{\bX\times_\bY \bX\ar@<-1ex>[rr]_{\pi_2}\ar@<+1ex>[rr]^{\pi_1}&& \bX\ar[rr]^{\f} &&\bY}
\] is exact in the category of formal superschemes.  
\end{theorem} 

Following \cite{takeuchi1977formal}, we introduce finiteness conditions on formal superschemes and their morphisms, including the notion of locally algebraic formal superschemes (see \Cref{sec.loc.alg}). We show that most properties of the usual case continue to hold in this setting. 
We also define the local superdimension of $\bX$ at a  point $x$, as well as the global superdimension of $\bX$, based on the theory of Krull superdimension in \cite{zubkov2022dimension}. We establish the following fiber-dimension type theorem.

\begin{theorem}
Let $\mathbf{f}:\bX\to\bY$ be a map of locally algebraic formal superschemes, $x\in|\bX|$ and $y=\mathbf{f}(x)$. Then 

\[
\sdim_{x, \overline{0}}\,\bX\leq\sdim_{y, \overline{0}}\,\bY+\sdim_{x, \overline{0}}\,\mathbf{f}^{-1}(y).
\]

If $\mathbf{f}$ is flat at $x$, then the equality holds. If furthermore $\O_y^*$ is regular and both $\O_x^*$ and $\O_y^*$ contain a field, then 

\[
\sdim_{x, \overline{1}}\,\bX\geq\sdim_{y, \overline{1}}\,\bY+\sdim_{x, \overline{1}}\,\mathbf{f}^{-1}(y).
\]  
\end{theorem}
It remains unclear when equality holds in the odd component (cf.~\cite[Theorem 7.2]{zubkov2022dimension}).

\subsection{Organization} 

The remaining of this paper is organized as follows.

In \Cref{sec.prel}, we present preliminaries on super(co)algebras. Our primary focus is on delivering superized versions of functional results for the study of formal schemes, including coradical filtrations, the wedge product, the Krull dimension, and the cofree coalgebra. Most of the results in this section follow the same as their non-super version, so several details are omitted.

In \Cref{sec.def.formal.super}, we introduce the concept of super $\K$-functor and formal $\K$-superscheme. We show they can approached from four different, but equivalent, perspectives, focusing on the super-coalgebraic point of view.

\Cref{sec.morphisms} is devoted to the study of morphisms between formal superschemes. We focus in particular on (faithfully) flat morphisms, and we prove a version of the faithfully flat descent theorem adapted to the super setting.

In \Cref{sec.loc.alg}, we explore the class of locally algebraic formal superschemes. We define and study the notion of (local) superdimension.

In \Cref{sec.forthcoming}, we conclude with a (brief) discussion of further research directions motivated by the theory developed in this work.  

%%%%%%%%%%%%%%%%%%%%%%%%%%%%%%%%%%%%%%%%%%%%%%%%

\section{Super(co)algebras and super(co)modules}\label{sec.prel}

This section provides an overview of the fundamental concepts of superalgebras and super-coalgebras, which play a pivotal role in the development of this work. We present ``superized" versions of several key results in \cite{sweedler1969hopf}, extending classical theorems to the super setting.

Throughout this work, unless explicitly stated otherwise, $\K$ denotes a field with characteristic not equal to 2. 

%%%%%%%%%%%%%%%%%%%%%%%%%%%%%%%%%%%%%%%%%%%%%%%%

\subsection{Supercommutative superrings} 

A \( \mathbb{Z}_2 \)-graded ring \( A = A\ev \oplus A\od \) is called a \emph{superring}. The set of homogeneous elements  of $A$ is \( A^h := A\ev \cup A\od \), and each \( x \in A^h \) is said to be \emph{homogeneous}. A nonzero homogeneous element \( x \) is called \emph{even} if \( x \in A\ev \), and \emph{odd} if \( x \in A\od \). Its \emph{parity} is \( |x| = \overline{i} \) if \( x \in A_{\overline{i}} \).

All superrings considered here are \emph{supercommutative}: this means \( A\ev \) is central in \( A \), i.e., $A\ev$ is contained in the center of $A$; and \( x^2 = 0 \) for all \( x \in A\od \). In particular, \( xy = (-1)^{|x||y|} yx \) for all \( x, y \in A^h \). If \( 2 \) is not a zero-divisor in \( A \), this condition is equivalent to supercommutativity. In particular, \( A\ev \) is a commutative ring.

Given two superrings \( A \) and \( B \), a \textit{morphism} \( \phi : A \to B \) is a \( \mathbb{Z}_2 \)-graded ring homomorphism. It is an \emph{isomorphism} if it is bijective, in which case we write \( A \simeq B \).

A \emph{superideal} of \( A \) is an ideal \( \mathfrak{a} \subset A \) such that \( \mathfrak{a} = (\mathfrak{a} \cap A\ev) \oplus (\mathfrak{a} \cap A\od) \). A superideal \( \mathfrak{p} \subset A \) is \emph{prime} (resp. \emph{maximal}) if \( A / \mathfrak{p} \) is an integral domain (resp. a field). Any such ideal is of the form \( \mathfrak{p} \oplus A\od \), where \( \mathfrak{p} \subset A\ev \) is a prime (resp. maximal) ideal of $A\ev$.

The \emph{canonical superideal} of \( A \) is \( \mathfrak{J}_A := A \cdot A\od = A\od^2 \oplus A\od \). The \emph{bosonic reduction} of \( A \) is the commutative ring \( \overline{A} := A / \mathfrak{J}_A \simeq A\ev / A\od^2 \).

For a more in-depth exploration of the fundamental theory of superrings, refer to \cite{westra2009superrings}.

\subsection{Krull superdimension of Noetherian superrings}  

We recall the definition of the Krull superdimension of a superring. For further details, see \cite{masuoka2020notion} and \cite{zubkov2022dimension}.

Let \( A \) be a Noetherian superring  (that is, it has the ACC on superideals) with \( \mathrm{Kdim}(A\ev) < \infty \). Let \( a_1, \ldots, a_s \in A\od \). For any subset \( I \subseteq \{1, \ldots, s\} \), write \( a^I := a_{i_1} \cdots a_{i_n} \), where \( I = \{i_1, \ldots, i_n\} \). The annihilator of \( a^I \) in \( A\ev \) is defined by
\[
\mathrm{Ann}_{A\ev}(a^I) := \{r \in A\ev \mid r a^I = 0\}.
\]
We say that \( a_{i_1}, \ldots, a_{i_n} \) form a \emph{system of odd parameters of length \( n \)} if
\[
\mathrm{Kdim}(A\ev) = \mathrm{Kdim}(A\ev / \mathrm{Ann}_{A\ev}(a^I)).
\]

The largest integer \( m \) for which there exists a system of odd parameters of length \( m \) is called the \emph{odd Krull superdimension} of \( A \), denoted \( \mathrm{Ksdim}\od(A) = m \). The \emph{even Krull superdimension} is defined as the Krull dimension of \( A\ev \), denoted \( \mathrm{Ksdim}\ev(A) \). The \emph{Krull superdimension} of \( A \) is the pair 
\[
\ksdim(A) := \mathrm{Ksdim}\ev(A) \mid \mathrm{Ksdim}\od(A).
\]

By \cite[Proposition~4.1]{masuoka2020notion}, a system of odd parameters can always be chosen among a set of generators of the \( A\ev \)-module \( A\od \). Since \( A \) is Noetherian, it follows from \cite[Lemma~1.4]{masuoka2020notion} that \( A\od \) is finitely generated as an \( A\ev \)-module. Consequently, the odd Krull superdimension is always finite.

\subsection{Super vector spaces} 

A \emph{super \( \K \)-vector space} is a \( \mathbb{Z}_2 \)-graded \( \K \)-vector space. We denote by $\mathsf{SVec}_\K$ the category of super $\K$-vector spaces and $\Z_2$-graded $\K$-linear maps.  Let $\Vs$ be a super $\K$-vector space, its dual $\Vs^*$ consists of the $\K$-linear maps $\Vs\to \K$. Let $S\subseteq \Vs$ be a subset. Define $$S^\perp:=\lbrace f\in \Vs^* \mid f(s)=0,  \text{ for all } s\in S\rbrace.$$ If \( \Vs \) and \(\Ws \) are super \( \K \)-vector spaces, there is a \emph{twisting map} 
\[
c_{\Vs, \Ws } : \Vs \otimes \Ws \to \Ws  \otimes \Vs,\ v \otimes w \mapsto (-1)^{|v||w|} w \otimes v.
\]

Let $\Vs\in\mathsf{SVec}_\K$. The \textit{super-dimension} of $\Vs$, denoted by $\sdim_\K(\Vs)$, is the pair $\dim_\K \Vs\ev\mid\dim_\K \Vs\od,$ where $\dim_\K(-)$ denotes the usual dimension of $\K$-vector space. If $$\dim_\K \Vs\ev+\dim_\K \Vs\od<\infty,$$ we say that $\Vs$ is \textit{finite-dimensional}. 

\subsection{Super-coalgebras} 

A \emph{super-coalgebra} over \( \K \) is a super \( \K \)-vector space \( \A \) equipped with morphisms of super $\K$-vector spaces $\Delta_\A: \A \to \A \otimes \A$ (the coproduct) and $\epsilon_\A : \A \to \K$ (the counit) such that \[ ( Id_\A \otimes_\K \epsilon_\A )\circ \Delta_\A = (\epsilon_\A \otimes_\K  Id_\A) \circ \Delta_\A =  Id_\A,\] where we are identifying \( \K \otimes_\K \A \cong \A \otimes_\K \K \cong \A \).

If there is no risk of confusion, we omit the subscript $\A$ in $\Delta_\A$ and $\epsilon_\A$; and the symbol $\otimes$ denotes the tensor product over $\K$. Furthermore, following \cite{sweedler1969hopf}, the coproduct is often written as
\[
\Delta(a) = \sum a_{(1)} \otimes a_{(2)}, 
\]

\noindent  or simply $\Delta(a) = a_{(1)} \otimes a_{(2)}$. We say that \( \A \) is \emph{coassociative} if \[ (\Delta \otimes  Id_\A) \circ \Delta = ( Id_\A \otimes \Delta) \circ \Delta.\] We always assume coassociativity. A super-coalgebra \( \A \) is called \emph{super-cocommutative} if \( \Delta = c_{\A, \A} \circ \Delta \).

Let \( (\A, \Delta_\A, \epsilon_\A) \) and \( (\B, \Delta_\B, \epsilon_\B) \) be super-coalgebras over \( \K \). A morphism of super $\K$-vector spaces \( \phi : \A \to \B \) is a \emph{morphism of super-coalgebras} if \[
\Delta_\B \circ \phi = (\phi \otimes \phi) \circ \Delta_\A\quad \text{  and }\quad  \epsilon_\B \circ \phi = \epsilon_\A.
\] 

\subsection{Sub super-coalgebras and the wedge product}\label{Sub.sec.2.5}

Let $\B$ be a sub super $\K$-vector space of $\A$. We say that $\B$ is a \textit{sub-super-coalgebra} of $\A$ if $\Delta(\B) \subseteq \B \otimes \B$. We say that $\B$ is a \textit{coideal} if $$\Delta(\B)\subseteq \B\otimes \A+\A\otimes \B \quad \text{and} \quad \epsilon(\B) = 0.$$
If $\B$ is a coideal, the quotient $\A/\B$ is a supercolagebra. It follows that  $\A\od$ is a coideal, and  $\A\ev \simeq \A / \A\od$. Then,  $\A\ev$ carries a coalgebra structure with coproduct $$\Delta_{\A\ev}(a\mod \A\od)=\Delta_A(a)\mod(\A\otimes \A\od+\A\od\otimes \A)$$ and counit $\epsilon_{\A\ev} (a\mod \A\od)=\epsilon_\A(a)$ (see \cite[Lemma 8.2.5 and Proposition 8.2.6]{westra2009superrings}).

Let $\A$ be a super-coalgebra and let $\Xs , \Ys $ be sub super $\K$-vector spaces of $\A$. The kernel of the composition \[
\xymatrix{\A\ar[r]^{\Delta}&\A\otimes \A\ar[r]& \A/\Xs \otimes \A/\Ys }
\]
is called the \textit{wedge} of $\Xs $ and $\Ys $ and is denoted by $\Xs \sqcap \Ys $. (The reader can check that the basic properties of ``$\sqcap$'' stated in  \cite[Proposition 9.0.0]{sweedler1969hopf} generalizes to the super case. Furthermore, note that we use the notation ``$\sqcap$'' instead of ``$\wedge$'' here, in order to avoid confusion with the exterior product.)  
 
We say that $\A$ is \textit{locally finite} if for any finite-dimensional sub super $\K$-vector spaces $\Xs $ and $\Ys $ of $\A$, the wedge $\Xs \sqcap \Ys $ is finite-dimensional.

The following proposition follows immediately from the definition:

\begin{proposition}
\ 

    \begin{enumerate}
        \item[\rm i)] Any sub-super-coalgebra of a locally finite super-coalgebra is locally finite.
        \item[\rm ii)]  Direct sums of locally finite super-coalgebras are locally finite. 
    \end{enumerate}    
\end{proposition}

\subsection{Super-comodules and rational supermodules} 

Let $(\A,\Delta, \epsilon)$ be a super-coalgebra. A \textit{left super-comodule} (over $\A$) is a super $\K$-vector space $\M$ together with a morphism of super $\K$-vector spaces $\psi:\M\to \M\otimes \A$ (the \textit{left-coaction}) such that   $$\psi\otimes Id_\A\circ\psi= Id_\M\otimes\Delta\circ\psi \quad\text{ and }\quad  Id_\M\otimes\epsilon\circ\psi= Id_\M.$$  In a similar way one defines \textit{right super-comodules}.  
 
Let $\M$ be a left super-comodule over $\A$ with structure map $\psi:\M\to\M\otimes \A$. If $\N\subseteq \M$ is a sub super $\K$-vector space of $\M$ and $\psi(\N)\subseteq \N\otimes \A$, then $\N$ a is a super-comodule with structure map $\psi|_\N$. We then say that $\N$ is a \textit{sub-super-comodule} of $\M$.
  
If $\A$ is a super-coalgebra, then the set $\A^*:=\hom_{\mathsf{Vec}_\K}(\A, \K)$  has a $\K$-\textit{superalgebra} structure. Furthermore, $\A^*$ is supercommutative provided that the super-coalgebra $\A$ is super-cocommutative (see \cite[Lemma 8.2.11]{westra2009superrings}). 
 
If $R$ is a superring, a \textit{left} (resp. \textit{right}) $R$-\textit{supermodule} is a $\Z_2$-graded left (resp. right) $R$-module. If $\lambda$ denotes the left action, a right action is defined by
\[
\rho(m, a) = (-1)^{|m||a|} \,  \lambda(a, m), \quad\text{ with }a\in R^h\text{ and }m\in M^h:=M\ev\cup M\od.
\]

With this definition, the left and right actions commute, and $R$ becomes both a left and right $R$-supermodule. In this case, we refer to $R$ simply as an $R$-\textit{supermodule}. Every supermodule considered in this paper is assumed to be equipped with these two compatible structures.  Now, if $M, N$ are $R$-supermodules, then the set $\homm_R(M, N)$ consisting of the morphisms $M\to N$ as $R$-modules carries a natural structure of $R$-supermodule (\cite[\S\,6.1.1]{westra2009superrings}).

Let $\A$ be a super-coalgebra. If $\psi:\A^*\otimes M\to M$ gives $M$ an $\A^*$-supermodule structure, consider the map 
    
    \[
    \rho_\psi:M\to\homm_{\A^*}(\A^*, M), m\mapsto\function{\rho_\psi(m)}{\A^*}{M}{a^*}{\psi(a^*\otimes m).}
    \]
We say that $(M, \psi)$ is a \textit{rational }$\A^*$-supermodule if $\rho_\psi(M)\subseteq M\otimes \A$. 

\begin{proposition}\label{prop:Super:2.1.3:In:Sweedler}
    Let $\A$ be a super-coalgebra, $(N, \phi)$ and $(M, \psi)$ be rational $\A^*$-supermodules.  

    \begin{enumerate}
        \item[\rm i)] A subset $P$ of $N$ is a sub-supermodule if and only if $\rho_\phi(P)\subseteq P\otimes \A$, i.e., $P$ is a sub super-comodule. In particular, a sub-supermodule of a rational supermodule is again rational. 
        \item[\rm ii)] Every cyclic sub-supermodule of $M$ is finite-dimensional, i.e., $M$ is generated by a single homogeneous element. 
        \item[\rm iii)] Any quotient supermodule of $M$ is again rational. 
        \item[\rm iv)] $\alpha:N\to M$ is a morphism of supermodules if and only if $\alpha$ is a morphism of super-comodules, i.e., $\psi\circ \alpha=\alpha\otimes  Id_\A\circ\phi$. 
    \end{enumerate}
\end{proposition}

\begin{proof}
    For the proof of  i), iii) and iv) see \cite[Propositions 9.2.4 and 9.2.5]{westra2009superrings}. For ii), let $m\in M^h$ and write $\rho_\phi(m)=\sum_{i=1}^n m_i\otimes a_i.$  Then  $$a^*\cdot m=\sum_{i=1}^n(-1)^{|m_i||a_i|} a^*(a_i)\,m_i\in\sum_{i=1}^n\K m_i.$$ Thus, ii) follows. 
\end{proof}

\subsection{Coradical filtrations} 

The coradical filtration of a super-coalgebra can be defined by viewing it solely as an ordinary coalgebra, ignoring the $\mathbb{Z}_2$-grading for this particular construction. However, let us briefly describe its construction. 

Let $\A$ be a super-coalgebra. We say that $\A$ is 
 \textit{irreducible} if any two nonzero sub-super-coalgebras of $\A$ have nonzero intersection. We call $\A$ \textit{simple} if it contains no sub-super-coalgebra except $0$ and $\A$ itself. The \textit{coradical} of $\A$, denoted $\cod\,\A$, is the (direct) sum of all simple sub-super-coalgebras of $\A$. For $n\geq0$, and a sub-super $\K$-vector space $\Xs $ of $\A$, we define $\scalebox{1.8}{$\sqcap$}^n\Xs $ inductively as follows:

\[\scalebox{1.8}{$\sqcap$}^0\Xs :=\{0\}, \quad\scalebox{1.8}{$\sqcap$}^1\Xs :=\Xs , \quad \ldots,\quad \scalebox{1.8}{$\sqcap$}^{n+1}\Xs =\left(\scalebox{1.8}{$\sqcap$}^n\Xs \right)\sqcap \Xs .
\]

Now, using \Cref{prop:Super:2.1.3:In:Sweedler} together with \cite[Proposition~6.4.5]{westra2009superrings}, one may proceed as in the proofs of \cite[Theorems~9.0.2 and~9.0.3]{sweedler1969hopf} to establish the following result:

\begin{proposition}\label{Dual:Krull} Let $M$ be a right super-comodule over $\A$. For $\{0\}\subseteq M$, we have 

\[
M=\bigcup_{n\geq0}\{0\}\sqcap\scalebox{1.8}{$\sqcap$}^{n}\cod\,\A.
\]
Additionally, if $N\subset M$ is a sub-supermodule such that $N\sqcap R=R$, then $N=M$. 
\end{proposition}

Note that for any sub-super-coalgebra $\B$, $\{0\}\sqcap \B=\B$. By \Cref{Dual:Krull}, we find that 

$$\A=\bigcup_{n\geq0}\scalebox{1.8}{$\sqcap$}^{n}\cod\,\A.$$

Now let $\A_n=\scalebox{1.8}{$\sqcap$}^{n+1}\cod\,\A$, for $n\geq0$. Then, we have that \[
\A=\bigcup\limits_{n\geq0}\A_n
\]
and 
$\A_n\subseteq\A_{n+1}, \text{ for all }n.$ We then call $\{\A_n\}$ a \textit{coradical filtration of $\A$}. If $\A_1$ is finite-dimensional, we say that $\A$ is of \textit{finite type}. 

\subsection{The super-coalgebra $A^\circ$} 

Let $A$ be a supercommutative $\K$-superalgebra with product $p$ and unit $\eta$. An ideal $\a$ of $A$ is called \textit{cofinite} if $A/\a$ if finite-dimensional. Let
\[
A^\circ:=\{f\in A^*\mid f\in \a^\perp\text{ for some cofinite superideal of }A \}.
\]
The map $p^*:A^*\to(A\otimes A)^*$ induce an isomorphism $$A^\circ\simeq A^\circ\otimes A^\circ$$ and $A^\circ$ becomes a super-coalgebra with counit given by  $\epsilon=\eta^*|_{A^\circ}$. Note that if $A$ is finite-dimensional, then $A^*=A^\circ$ (see \cite[\S1.3]{heyneman1969affine} for further details).

\begin{remark}\label{remark.adj.*.and.circ}
    The functors $(-)^\circ$ and $(-)^*$ are adjoint to one another (\cite[VI, \S\,0]{sweedler1969hopf}). 
\end{remark}

The following proposition will be useful in the sequel. 

\begin{proposition}\label{prop.equiv.fint.type} Let $\A$ be super-cocommutative. Denote by $\A_n$ the subobjects in the coradical filtration of $\A$. The following conditions are equivalent.

\begin{enumerate}
    \item[\rm i)] $\A$ is of finite type.
    \item[\rm ii)] $\A_n$ is finite-dimensional for all $n$. 
    \item[\rm iii)] $\A$ is locally finite and $\A_0$ is finite-dimensional. 
    \item[\rm iv)] $\A_0$ is finite-dimensional and $\A$ is reflexive, i.e., $\A\simeq (\A^*)^\circ$.
    \item[\rm v)] $\A^*$ is Noetherian.
    \item[\rm vi)] The descending chain condition, DCC, for sub-super-coalgebras of $\A$.
\end{enumerate}
\end{proposition}

\begin{proof} We will show the following:

{ 
    \[
\xymatrix{&\text{iv)}\ar@{=>}[ld]&\\
\text{iii)}\ar@{<=>}[r]&\text{i)}\ar@{=>}[d]\ar@{=>}[u]\ar@{<=>}[r]&\text{ii)}\\
&\text{v)}\ar@{<=>}[r]&\text{vi)}\ar@{=>}[ul]}
\]}

The proof of \cite[Proposition 1.4.1]{takeuchi1974tangent} shows the equivalence between i) and ii), and also can be used to deduce that i) (and hence ii)) implies v).  

To show that i) (hence ii)) implies iv), one may proceed as in \cite[pp.26-28]{takeuchi1974tangent}.  

To show that i) (or equivalently that ii)) implies iii), one may proceed as in \cite[Proposition 2.4.3]{heyneman1974reflexivity}, just keeping in mind that the only group-like elements are even. (Indeed, $a$ is \textit{group-like} if $\Delta(a)=a\otimes a$ and $\epsilon(a)=1$, so that no odd element may be group-like.) The other implication is clear, so we find that i), ii), and iii) are equivalent. 

To prove that vi) implies i) one proceeds as in \cite[Lemma 2.3.1]{heyneman1974reflexivity}. 

A sub-super-vector space $\Vs$ of $\A$ is a sub-super-coalgebra if and only if  $\Vs^\perp$ is a superideal of $\A^*$. Thus, v) is equivalent to vi) because for sub super $\K$-vector spaces $\Vs$ and $\Ws $ of $\A$, $\Vs\subseteq \Ws $ is equivalent to $\Vs^\perp\supseteq \Ws ^\perp$. 

Finally, as in \cite{heyneman1974reflexivity}, any reflexive super-coalgebra is locally finite, then iv) implies iii), and the proof of the proposition is complete. 
\end{proof}

\begin{proposition}\label{tensor.fin.type}
    If $\A, \B$ are  of finite type, then $\A\otimes\B$ is of finite type. 
\end{proposition}

\begin{proof}  
Proceeding as in \cite[2.3.13]{heyneman1974reflexivity}, we see that for each $n \geq 0$, the following holds:
\[
(\A \otimes \B)_n \subseteq \sum_{r+s=n} \A_r \otimes \B_s.
\]
Thus,
\[
(\A \otimes\B)_1 \subseteq \A_0 \otimes \B_1 + \A_1 \otimes \B_0.
\]
Now, since $\A_0$ and $\B_0$ are finite-dimensional by \Cref{prop.equiv.fint.type}, and $\A_1$ and $\B_1$ are finite-dimensional by assumption, we see that the right-hand side of the inclusion above is finite-dimensional. Therefore, $(\A \otimes \B)_1$ is finite-dimensional as well, and the proposition follows.
\end{proof}

\subsection{Cotensors} 

Let $\A$ be a super-coalgebra and $\M$ a right super-comodule over $\A$ and $\N$ a left super-comodule over $\A$. The \textit{cotensor product} $\M\square_\A\N$ is the kernel of the map 

\[
\xymatrix{\M\otimes\N\ar@<-1ex>[rr]_{ Id_\A\otimes\theta_\N}\ar@<+1ex>[rr]^{\theta_\M\otimes Id_\A}&&\M\otimes \A\otimes \N,}
\]
where $\theta_\M$ and $\theta_\N$ are respectively the structure maps of $\M$ and of $\N$. It is a super-comodule over $\A$. The cotensor product endows the category of super-comodules over $\A$, denoted by $\mathsf{SMod}^\A$, with the structure of a monoidal category (see \cite{takahashi2024quotients} for further details). 

\subsection{Cofree super-coalgebras}\label{cofree.coalgebra.sect}

If $\Vs$ is a super vector space, a pair $(\A, \pi)$ with $\A$ a super-coalgebra and $\pi:\A\to \Vs$ is called a \textit{cofree super-coalgebra} on $\Vs$ in case for any super-coalgebra $\B$ and map $\theta:\B\to\Vs$ there is a unique morphism of super-coalgebras $F:\B\to\A$ such that $\theta=\pi\circ F$. If $\A$ exists, it is unique up to a unique super-coalgebra isomorphism. We denote the cofree super-coalgebra on $\Vs$ by $\mathfrak{Cof}(\Vs)$. The existence and uniqueness (up to isomorphism) of \( \mathfrak{Cof}(\Vs) \) are guaranteed by the following proposition.

\begin{proposition}
    If $\Vs\in\mathsf{SVec}_\K$, the cofree super-coalgebra on $\Vs$ always exists and is unique up to isomorphism. 
\end{proposition} 

\begin{proof}
    Take \( \mathbb{C}= \K \) in \cite[Proposition~3.5]{masuoka2011quotient}, or see \cite[Proposition~3.11]{HOSHI202028}.
\end{proof}

The following proposition will be needed later.

\begin{proposition}\label{lemma.emd.fn.tp}
    Any super-coalgebra of finite type can be embedded into a $\mathfrak{Cof}(\Vs)$ for some finite-dimensional super vector space  $\Vs$.
\end{proposition}

\begin{proof}   
  See \cite[p.1497]{takeuchi1977formal}. 
\end{proof}

\subsection{Krull superdimension of super-coalgebras} 

Let $\A$ be a super-coalgebra. A nonzero sub-super-coalgebra $\B$ of $\A$ is said to be \textit{coprime} if whenever $\B\subseteq \mathbf{X}\sqcap \mathbf{Y}$, then $\B\subseteq \mathbf{X}$ or $\B\subseteq \mathbf{Y}$, for any sub-super-coalgebras $\mathbf{X}$ and $\mathbf{Y}$ of $\A$. This is equivalent to requiring that $\A^\perp$ is a prime ideal. It is not hard to show that coprime sub-super-coalgebras of $\A$  are purely even: they are contained in $\A\ev$. 

Let $\A$ be a super-cocommutative super-coalgebra of finite type. Recall from \Cref{Sub.sec.2.5} that $\A\ev$ has a coalgebra structure. Suppose that with this structure, we have $\kdim(\A^*\ev)<\infty$. We define the \textit{Krull superdimension} of $\A$ to be the Krull superdimension of $\A^*$, and is denoted by $\ksdim\,\A$. 

By \Cref{prop.equiv.fint.type}, $\A^*$ is Noetherian and in particular, $(\A^*)\ev\simeq(\A\ev)^*$ is so. Thus, $\A\ev$ is of finite type and hence $\kdim\,\A\ev=\kdim\,(\A\ev)^*$. In particular, $\kdim\,\A\ev=\ksdim\ev\,\A$.

%%%%%%%%%%%%%%%%%%%%%%%%%%%
%%%%%%%%%%%%%%%%%%%%%%%%%%%

\section{Definition of formal superscheme}\label{sec.def.formal.super}

In this section, we define the notion of a formal superscheme and derive some basic properties. In particular, we show that this notion, as its purely even counterpart, admits several equivalent definitions. Since some of the required arguments are identical to those in the commutative setting, we will occasionally omit the details. 

\subsection{(Formal) Super $\K$-functors} 

Let $\alg$ be the category of supercommutative $\K$-superalgebras and $ \sets$ the category of sets. 

\begin{definition}
    A (covariant) functor $\alg\to\sets$ is called a \textit{super $\K$-functor}. We denote the category of super $\K$-functors by $\SF$.
\end{definition}

\begin{definition}
\ 
 
\begin{itemize}
    \item[\rm i)] Let $\mathcal{B}:\mathsf{SAlg}_\K\rightarrow \mathsf{Alg}_\K$ be the functor that assigns to each superalgebra $A$ its bosonic reduction  $\overline{A}$.
    \item[ii)] We define the functor $\mathcal{S}:\mathsf{Alg}_\K\rightarrow\mathsf{SAlg}_\K$ on objects as $R\mapsto \mathcal{S}(R)$ with $\mathcal{S}(R)_{\overline{0}}=R$ and $\mathcal{S}(R)_{\overline{1}}=0$. 
    
    \item[\rm iii)] Let $\bX$ be a super $\K$-functor. The \textit{bosonic reduction of } $\bX$, denoted by $\bX_{\mathrm{bos}}$, is the composition $\bX_{\mathrm{bos}}:= \bX\circ\mathcal{S}$. 
\end{itemize}
\end{definition} 

\begin{remark}
\ 
 
\begin{itemize}
    \item[i)] The functor $\mathcal{S}$ is right-adjoint to $\mathcal{B}$.
    \item[ii)] The process of taking a super $\K$-functor to its bosonic reduction can be seen as a functor 
    \[
    \mathfrak{B}:\mathsf{SFunt}_\K\rightarrow \mathsf{Funt}_\K, 
    \]
    and we also have a functor 
    \[
    \mathfrak{S}:\mathsf{Funt}_\K\rightarrow \mathsf{SFunt}_\K
    \]
    that maps $\bX$ to the composition $\bX\circ\mathcal{B}$ of a $\K$-functor $\bX$ with $\mathcal{B}$.
\end{itemize} 
\end{remark}

\begin{proposition}
   Notation as above. The functor $\mathfrak{B}$ is left-adjoint to $\mathfrak{S}$. 
\end{proposition}

\begin{proof} Let $T:\mathfrak{B}(\bX)\rightarrow \bY$ be a morphism. Let $A$ be a superalgebra and $f_A:A\rightarrow \mathcal{S}\circ\mathcal{B}(A)$ be the adjunction map. Consider the following composition: 
  $$\bX(A)\rightarrow\bX(\mathcal{S}\circ\mathcal{B}(A))=\mathfrak{B}(\bX)(\mathcal{B}(A))\rightarrow\bY(\mathcal{B}(A))=\mathfrak{S}(\bY)(A),$$
  where the first map is induced by $f_A$ and the second is $T_{\mathcal{B}(A)}$. This maps is functorial as $f_A$ and  $T_{\mathcal{B}(A)}$ are. Similarly, we construct a functorial inverse. 
\end{proof}

\begin{definition}
    Let $A\in\alg$. We define the super $\K$-functor $\sp_\K\,A$ (or simply $\sp\,A)$ as

\begin{align*}
    B&\mapsto\hom_{\alg}(A, B),\\
    (\varphi:B\to C)&\mapsto\function{(\sp\,A)(\varphi)}{\hom_{\alg}(A, B)}{\hom_{\alg}(A, C)}{\text{\hspace{1.5cm}}\quad\psi}{\varphi\circ\psi.\quad\text{\hspace{1cm}}}
\end{align*} 
An \textit{affine $\K$-superscheme} is a super $\K$-functor isomorphic to some $\sp\,A$. In such case, $A$ is called the \textit{coordinate superalgebra} of the superscheme $\bX=\sp\,A$, and is also denoted by $\K[\bX]$ (\cite[p.719]{zubkov2009affine}). 
\end{definition}

Let $\falg$ be the full subcategory of $\alg$ consisting of the finite-dimensional $\K$-superalgebras.
 
\begin{definition}
A \textit{formal super $\K$-functor} is a  covariant functor $\falg\to\sets$. The category of formal super $\K$-functors is denoted by $\SfF$.   
\end{definition}

\begin{remark} The inclusion functor $\falg\to\alg$ gives rise a functor  

\[
\function{\widehat{ }\,}{\SF}{\SfF\quad\text{\hspace{1.5cm}}}{\bX}{\function{\widehat{\bX}}{\falg}{\sets}{A}{\bX(A),}}
\]
called the \textit{completation functor}. The functors $\mathcal{S}$ and $\mathcal{B}$ may be restricted to $\mathsf{Algf}_\K$ and $\mathsf{SAlgf_\K}$, respectively. Hence, both bosonic reduction $\mathfrak{B}$ and functor $\mathfrak{S}$ may be defined on formal super $\K$-functors. By abuse of notation, we retain the names $\mathfrak{B}$ and $\mathfrak{S}$ for them, which remain adjoint to each other.  \end{remark}
 
\begin{definition}
Let $A\in\falg$. Define the super formal $\K$-functor $\spf_\K$ (or simply $\spf$) as

\[
\function{\spf_\K}{\falg}{\sets\quad\text{\hspace{1.5cm}}}{R}{\hom_{\falg}(A, R).}
\] or equivalently, $\spf A:=\widehat{\sp\,A}.$
\end{definition}

The following proposition follows from Yoneda's lemma (\cite[Lemma 1.1]{masuoka2011quotient}).   

\begin{proposition}\label{Prop.Hom.SSpf.X.is.X.A}
    If $\bX\in\SfF$ and 
    $A\in\falg$, then $\mathrm{Hom}_{\SfF}(\spf A, \bX)\simeq\bX(A)$.
\end{proposition}

Let $\K[X|Y]$ consists of polynomials in the even indeterminate $X$ and odd indeterminate $Y$, with coefficients in $\K$. If $\bX$ is  a formal super $\K$-functor and $\mathbb{A}^{1\mid 1}_\K:=\sp\,\K[X|Y]$, a \textit{regular function} on $\bX$ is a morphism $\f:\bX\to\mathbb{A}^{1\mid 1}$. We denote $\O(\bX)$ the set of regular functions on $\bX$, which is endowed with an obvious $\K$-superalgebra structure. 

\subsection{Superschemes}

\begin{definition}
    Let  $E$ be a set of regular functions on $\bX$. We define the two subfunctors $\V(E)$ and $\D(E)$ of $\bX$ as

    \begin{align*}
     \V(E)(A)&=\{x\in\bX(A)\mid f(x)=0\text{ for all }f\in E\},\text{ and }\\   
    \D(E)(A)&=\{x\in\bX(A)\mid f(x)\text{ (for }f\in E)\text{ generate the unit ideal of }A\}.
    \end{align*} 
\end{definition}

\begin{remark}
    Let $\bX, \bY$ be super $\K$-functors and  $g:\bY\to\bX$ a morphism of super $\K$-functors. Let $F=\{f\circ g\mid f\in E\}\subseteq\O(\bY)$. Then 

\begin{align*}
g^{-1}(\V(E))&=\V(F),\text{ and }\\ 
g^{-1}(\D(E))&=\D(F).
\end{align*}  

Moreover, if $\bX$ is an affine $\K$-superscheme, say $\bX=\sp\,A$ and $I\subseteq \O(\bX)^h$, then for all $R\in\alg$, we have 

\begin{align*}
    \V(I)(R)&=\{\varphi\in\hom_{\alg}(A, R)\mid\varphi(I)=0\} \simeq\hom_{\alg}(A/IA, R).
\end{align*}
In particular, $\V(\a)=\sp  (A/\a)$, for any superideal $\a$ of $A$. Also, if $\bX=\D(\{f\})$, then

\begin{align*}
    \bX_f(A)&:=\D(\{f\})(A)\\
    &=\{\varphi\in\hom_{\alg}(A, R)\mid\varphi(f)\text{ is invertible}\}\\
    &\simeq\hom_{\alg}(A[f\ev^{-1}], R);
\end{align*} 
that is, $\K[\bX_f]=A_f=\K[\bX]_{f}$ (cf.  \cite[\S2]{zubkov2009affine} and \cite[I, \S3]{demazure2006lectures}).
\end{remark}

\begin{definition}
    A subfunctor $\bY$ of $\bX$ is \textit{closed} (resp. \textit{open})  if for any morphism $g:\mathfrak{T}\to\bX,$ where $\mathfrak{T}$ is an affine superscheme, the subfunctor $g^{-1}(\bY)$ of $\mathfrak{T}$ is of the form $\V(E)$ (resp. $\D(F)$).

    A collection of open subfunctors $\{\bY_i\}_{i\in I}$ of a super $\K$-functor $\bX$ is called an \textit{open covering} whenever $$\bX(A)=\bigcup\limits_{i\in I}
\bY_i(A)$$ for any field extension $A$ of $\K$.  
 \end{definition}

As customary, the symbol $(\alg)^{\circ}$ denotes the 
opposite category of $\alg$ which is isomorphic to the category of affine superschemes over $\K$. 

\begin{definition}[See p.142 in \cite{masuoka2011quotient}]
    We define a \textit{Grothendieck topology} $T_{\mathrm{loc}}$ in $(\alg)^{\circ}$ as follows. A covering in $T_{\mathrm{loc}}$ is defined to be a collection of finitely many morphisms $$\{\sp\, R_{f_i}\to \sp\, R\}_{1\leq i\leq n},$$ where $R\in \alg$ and the $f_1, \ldots, f_n\in R\ev$ are such that $\sum_{1\leq i\leq n}R\ev f_i=R\ev;$ that is, the $f_i$ are a \textit{partition of the unity}.  Note that each $\sp\, R_{f_i}\to \sp\, R$ is an isomorphism onto $\D(R_{f_i})$ and the $\D(R_{f_i})$ form an open covering of $\sp\, R$. 
\end{definition}

\begin{definition}[Definition 3.2 in \cite{masuoka2011quotient}]\label{def-superscheme}\textit{ }

\begin{enumerate}
    \item[i)] A sheaf $\bX$ on $T_{\mathrm{loc}}$ is called a \emph{local functor}. (Then, any affine superscheme is a local functor.) 
    \item[ii)] A local super $\K$-functor $\bX$ is called a \emph{superscheme} provided $\bX$ has an open covering $\{\bY_i\}_{i\in I}$ with $\bY_i\simeq \sp\, A_i$, with $A_i\in \alg$. The full subcategory of all superschemes in $\SF$ is denoted by $\ss$. 
    \item[iii)] A superscheme $\bX$ is said to be \emph{Noetherian}, if it has an open covering $\{\bY_i\}_{i\in I}$ with $\bY_i\simeq \sp\, A_i$, as above, such that $I$ is finite, and each $A_i$ is Noetherian. Observe that an affine superscheme $\sp\, A$ is Noetherian if and only if $A$ is Noetherian.  
    \item[iv)] A $\K$-superscheme $\bX$ is said to be \textit{finite} if it is affine and $\O(\bX)$ is finite-dimensional.
\end{enumerate}

\end{definition}

 \begin{remark}\label{rem.3.15}  By \Cref{Prop.Hom.SSpf.X.is.X.A}, the completation functor, restricted to the full subcategory consisting of finite $\K$-superschemes, is fully faithful. Hence, we may see the category of finite $K$-superschemes as a full subcategory of $\SF$ or of $\SfF$ (\cite[p.8]{demazure2006lectures}).
\end{remark}
 
\begin{definition}[See pp.~8 and 12 in \cite{demazure2006lectures}, and p.~4090 in \cite{takahashi2024quotients}] \label{def.formal.schem}  A \textit{formal $\K$-superscheme} $\bX$ is a formal super $\K$-functor isomorphic to the inductive limit of a filtered inductive system of finite affine $\K$-superschemes. In other words, there is a filtered inductive system of finite affine $\K$-superschemes $(\bX_\lambda, u_{\lambda\mu})$, where each $\bX_\lambda$ is of the form $\spf\, A_\lambda$ with $A_\lambda \in \falg$, and the transition morphisms $u_{\lambda\mu} : \bX_\lambda \to \bX_\mu$ for $\lambda < \mu$ are closed immersions (i.e., the corresponding morphisms $A_\mu \to A_\lambda$ of $\K$-superalgebras are surjective).  
\end{definition}  

\subsection{The underlying topological space} 

We now compare the previous constructions with  geometric superspaces (cf. \cite[\S4]{demazure2006lectures}). We first recall the following from \cite[\S\S1-5]{masuoka2011quotient}. For the definition of geometric superspace, see \cite[\S4]{masuoka2011quotient}. 

Let $\mathcal{A}$ be a category. For an object $A \in \mathcal{A}$, we define the functor $h^A$  by $B \mapsto \mathrm{Hom}_{\mathcal{A}}(A, B)$ for each $B \in \mathcal {A}$. The category of covariant functors from $\mathcal{A}$ to $\mathsf{Sets}$ is denoted by $\mathsf{Sets}^{\mathcal{A}}$. 

Let $F\in \mathsf{Sets}^{\mathcal{A}}$. Now, consider the category $\mathcal{M}_F$ whose objects are 
pairs $(A, a)$, where $A\in \mathcal{A}$ and $a\in F(A)$; and morphisms $(A, a) \to (B, b)$ given by morphisms $\phi$ in $\mathrm{Hom}_{\mathcal{A}}(A, B)$ with 
$F(\phi)(a)=b$. Let $\mathcal{V}$ be the category of geometric superspaces. Now we define the functor $d_\bX : (\mathcal{M}_\bX)^{\circ} \to\mathcal{V}$ as follows 

\begin{align*}
    (A, a) &\mapsto\sspec\ A\\
    \phi:A\to B&\mapsto \function{\sspec(\phi)}{\sspec\ B}{\sspec A}{\p}{\phi^{-1}(\p).}
\end{align*}

By \cite[Lemma 4.3]{masuoka2011quotient}, $\underrightarrow{\lim} \ d_\bX$ exists and corresponds to a geometric superspace. It is called a {\it geometric realization} of $\bX$ and denoted by $|\bX|$. 

Given $A\in\alg$, we define $\bX^{ \lozenge}(A):=\mathrm{Hom}_{\mathcal{V}}(\sspec\,A, \bX)$. Thus, we have the following comparison theorem: 

\begin{proposition}[Theorem 5.14 in \cite{masuoka2011quotient}]
    The functors $\bX\to|\bX|$ and $\bX\mapsto\bX^{\lozenge}$ define equivalences of $\mathsf{SSch}_\K$ and the category of full subcategory of $\mathcal{V}$ consisting of geometric superschemes, $\mathcal{SV}$. Furthermore, these functors are quasi-inverses of each other.
\end{proposition}

Let $\bX$ be a (formal) $\K$-superscheme. By abuse of notation, we use the symbol $|\bX|$ for the topological space of $\bX$.

\begin{remark}
   Let $\bX=\underrightarrow{\lim}\,\bX_i$ be a formal superscheme. The adjunction between $\mathfrak{S}$ and  $\mathfrak{B}$, implies that $\bX_{\mathrm{bos}}=\underrightarrow{\lim}\,\bX_{i,\, \mathrm{bos}}$, so $\vert\bX_{\mathrm{bos}}\vert=\lim \vert \bX_{i,\mathrm{bos}}\vert.$ 
\end{remark}
 
\subsection{The definition of $\sp^*$}\label{sect.3.4}  Let \(\A \) be a super-cocommutative super-coalgebra and consider any element \( R \in \falg \). Note that \( R \otimes \A \) is a super-coalgebra (over $R$). We define  
\[
G_R(R \otimes \A) = \left\{ u \in (R \otimes \A)\ev \mid \Delta(u) = u \otimes_R u \text{ and } \epsilon(u) = 1 \right\}.
\]  
These elements are called \textit{group-like}. (Note that we must only allow even elements to be group-like, because the coproduct and the counit must preserve the grading; thus, an odd or non-homogeneous element cannot be group-like under this definition.) If $R=\K$, we simply write $G(\A)$. The assignment \[ R \mapsto G_R(R \otimes \A)\] gives rise to a formal superscheme $\sp^*\A$. 

Given $A, R\in\falg$, consider $\Delta_{A^*}$ and $\epsilon_{A^*}$ the structure maps of $A^*$. Then, we may extend then to $R$-linear maps $\Delta:R\otimes A^*\to(R\otimes A^*)\otimes_R(R\otimes A^*)$ and $\epsilon:R\otimes A^*\to R$, respectively. The maps $\Delta$ and $\epsilon$ endow $R\otimes A^*$ with a  $R$-super-coalgebra structure. Furthermore, each $\Z_2$-graded linear map $A\to R$ corresponds uniquely with an element in $(R\otimes A^*)\ev$ such that $\Delta(u)=u\otimes u$ and $\epsilon(u)=1$. Thus, there is functorial isomorphism $$\sp\,A(R)\simeq\{u\in (R\otimes A^*)\ev\mid\Delta(u)=u\otimes_R u\text{ and }\epsilon(u)=1\}=\sp^*A^*(R).$$
(This identification is only possible because $R$ is finite-dimensional. This is one of the main points for defining $\sp^*$ on $\falg$ instead of $\alg$.) As in the usual non-super case (\cite[Theorem 1.1]{takeuchi1977formal}), the following proposition holds: 

\begin{proposition} \label{3.16}
The functor $\A\mapsto\sp^*\A$ is an equivalence from the category of super-coalgebras to that full subcategory of the category of super $\K$-functors which consist of all formal superschemes. 
\end{proposition}

Observe that if $\O_*$ is the quasi-inverse functor of $\sp^*$, then $\bX\simeq\sp^*(\O_*(\bX))$ for any formal superscheme $\bX$.

It is worth noting that the functor $\sp^*$ can also be defined on the category of super-coalgebras. Specifically, given a super-coalgebra $\A$, the functor $\sp^*$ maps any super-coalgebra $\B$ to the set $\hom_{\mathsf{SCoalg}_\K}(\A, \B)$ (that can also be identified with the group-like elements in $(\A\otimes \B)\ev$). From this point of view, there also exists a natural relationship between $\sp$ and $\sp^*$, as stated in the following proposition, whose proof follows the same reasoning as that of \cite[Theorem 6.0.5]{sweedler1969hopf} (cf. \Cref{remark.adj.*.and.circ}).

\begin{proposition}
    If $R$ is a supercommutative superalgebra and $\A$ is a super-cocommutative super-coalgebra, then we have a bijection 
   $\sp\,R(\A^*)\simeq\sp^* R^\circ(\A).$
\end{proposition}

\begin{remark} 
  The bosonic reduction can be seen in terms of super-coalgebras. Indeed, consider   $\bX=\sp^* \A$. Then $\bX_{\mathrm{bos}}=\sp^*(\,\A/\A\od\,)$, because $\A=\lim R^*_i$ implies that $\A/\A\od=\lim R^*_{i}/(R_i^*)\od$. 
\end{remark}

\begin{remark} 
    We have seen that the notion of a formal superscheme can be approached both from the perspective of formal super $\K$-functors and from that of super-coalgebras. There are also two additional approaches: one characterizes formal superschemes as those formal super $\K$-functors which commute with finite projective limits, and the other is based on profinite topological $\K$-superalgebras. Indeed, a topological $\K$-superalgebra that is the inverse limit of discrete quotients which are finitely generated $\K$-superalgebras is said to be \textit{profinite}. If $A$ is a profinite $\K$-superalgebras, we can define $\mathrm{PSSpf}\,A(R)$ to be the set of all continuous morphisms of topological $\K$-superalgebras from $A$ to $R$, for any $R\in\falg$ endowed with the discrete topology. Then, if $\{A_i\}$ is the family of discrete finitely generated quotients of $A$ defining its topology, we have $$\mathrm{PSSpf}\,A(R) = \varinjlim\, \mathrm{PSSpf}\, A_i(R),$$ and $\mathrm{PSSpf}\,A$ is a formal $\K$-superscheme. As in the purely even case (\cite[I, \S6, b)]{demazure2006lectures}), the assignment $ A \mapsto \mathrm{PSSpf}\, A$  defines an anti-equivalence between the category $\salg$ of profinite $\K$-superalgebras and the category of formal $\K$-superschemes. Since we will not make use of these alternative approaches here, we omit further details.
\end{remark} 

\begin{definition}
    We define 

    \begin{align*}
        \sp^* \A_1\times_{\sp^*\A}\sp^*\A_2&:=\sp^*(\A_1\square_\A \A_2),\\
        \sp^* \A_1\times\sp^*\A_2&:=\sp^*(\A_1\otimes \A_2), \quad \text{and}\\
        \coprod\sp^*\A_i&:=\sp^* \sum \A_i.
    \end{align*}  
\end{definition} 

\begin{definition}
    A formal superscheme $\bX$ is said to be \textit{local} if it is isomorphic to a $\spf A$, where $A$ is a local superring. 
\end{definition}

Using \cite[Corollary 5.2.4]{westra2009superrings}, we find the following.

\begin{proposition}[\cite{demazure2006lectures}, p.15]
    Any formal superscheme is a direct sum of local formal superschemes.  
\end{proposition} 

\begin{definition}
    Let $F|\K$ be a field extension. We define the \textit{base-change functor} by $$(\sp^*\A)\otimes_\K F=\sp^*(\A\otimes_\K F).$$
\end{definition}
 
Thus, the fiber product of two formal superschemes is a formal superscheme and the scalar extension functor preserves formal superschemes.

\subsection{Irreducible components} 

Let $\A$ be a super-coalgebra. The \textit{irreducible components} $\{\A_i\}_{i\in I}$ of $\A$ can be defined as in the usual case (\cite[Definition 9.2.16]{westra2009superrings}). Thus, by \cite[Theorem 9.2.17 ii), iii)]{westra2009superrings}, 
\[
\A=\bigoplus_{i\in I} \A_i\]
and we find that \[|\sp^*\A|=\coprod_{i\in I}|\sp^*\A_i|.
\]

Now let $R$ be a local and finite-dimensional superalgebra. Then $\overline{R}$ is also local and finite-dimensional, with dimension given by $\dim_\K R - \dim_\K \mathfrak{J}_R$. In particular, the correspondence $\sspec\, R \simeq \spec\, \overline{R}$ implies that the space $\sspec\, R$ consists of a single point. Consequently, the topological space $|\sp^* \A_i|$ consists of exactly one point. Hence, there is a bijection between the set $|\sp^* \A|$ and the family $\{\A_i\}_{i \in I}$. Following \cite[p.~1488]{takeuchi1977formal}, we refer to each $\sp^* \A_i$ as an \textit{irreducible component} of the formal superscheme $\sp^* \A$.

\newpage

If $\bX$ is a formal superscheme and $x \in |\bX|$, we denote by $\bX_x$ the irreducible component of $\bX$ that contains $x$. Accordingly, we may write
\[
\bX = \coprod_{x \in |\bX|} \bX_x.
\]
We define the super-coalgebra $\O_x := \O_*(\bX_x)$. Its coradical is denoted by $\kappa(x)$. The formal sub-superscheme $\sp^*(\kappa(x))$ of $\bX_x$ is denoted by $\{x\}$.

\section{Morphisms of formal superschemes}\label{sec.morphisms}

In this section, we study the basic properties of morphisms of formal superschemes.

Let $\bX$ and $\bY$ be formal superschemes. A map of formal superschemes $\f:\bX\rightarrow\bY$ induces a continuous function $$\vert \f\vert:\vert \bX\vert\rightarrow \vert \bY\vert$$ and a morphism $$\O_*(\f):\O_*(\bX)\rightarrow \O_*(\bY)$$ of super-coalgebras. Furthermore,  $\vert \bX\vert=\vert\bX_{\mathrm{bos}}\vert$, and we have the following.  

\begin{proposition}\label{f_eq_f_red}
    Let $\f:\bX\rightarrow \bY$ be a morphism of formal superschemes. Then $\vert \f\vert=\vert \f_{\mathrm{bos}}\vert$. 
\end{proposition}

\subsection{Immersions}

\begin{definition} 
    A morphism $\f:\bX\rightarrow\bY$ of formal superschemes is said to be 

    \begin{enumerate}
        \item[i)] \textit{affine} if for every $\K$-superalgebra $A$ and every morphism $\f:\sp\,A\to \bY$ the fiber product $\sp\,A\times_\bY\bX$ is an affine superscheme (cf. \cite[I, \S2, 5.1]{demazure1980introduction}).  
        \item[ii)] a \textit{closed immersion} if it is a closed immersion as a map of $\bW$-functors. In other words, if it is affine and for every $\K$-superalgebra $A$ and any morphism  $\sp\,A\to \bY$, the induced map $\O_*(\f_{\sp\,A}):\O_*(\sp\,A)\to\O_*(\sp\,A\times_{\bY} \bX)$ is surjective  (cf. \cite[I, \S2, 6.1]{demazure1980introduction}).
        \item[iii)] \textit{surjective} (resp. \textit{strictly surjective}) if $\vert \f\vert$ (resp. $\O_*(\f)$) is surjective. 
    \end{enumerate}
\end{definition}

Thus, by \Cref{f_eq_f_red} we find the following.

\begin{proposition}
    A map of formal superschemes $\f$ is surjective if and only if $\f_{\mathrm{bos}}$ is.
\end{proposition}

In order to obtain some useful characterizations of the properties introduced above for morphisms of formal superschemes, we will need the following lemma.

\begin{lemma}\label{lemma.equiv.surj.epim}
\ 

\begin{enumerate}
    \item[\rm i)] Let $\phi:A\rightarrow B$ be a morphism of $\K$-superalgebras. If $B$ is finitely generated $A$-supermodule, then $\phi$ is an epimorphism in $\alg$ if and only if it is surjective as a map of sets. 
    \item[\rm ii)]  A map $\psi:\A\rightarrow \B$ of super-cocommutative super-coalgebras is a monomorphism in the category of super-coalgebras if and only if it is injective as a map of sets.
\end{enumerate}     
\end{lemma}

\begin{proof}
    We first show the following.

\textbf{Claim.} \textit{Let $R$ be a superring. Let $M$ be an $R$-supermodule, $N$ a finitely generated $R$-supermodule (that is, for some $n_1, \ldots, n_k\in N^h$, we have $N=n_1R+n_2R+\cdots+n_kR$), and $\phi:M\to N$ a morphism of $R$-supermodules. Then, $\phi$ is surjective if and only if for all maximal ideal $\m$ of $R$, the induced morphism $M/\m M\to N/\m N$ be surjective. }

The claim can be shown in a similar manner to \cite[II, \S3, n°3, Proposition 11]{bourbaki1998commutative}, so we only sketch the proof.

The first thing to note is that if $M$ is an $R$-supermodule and $\mathfrak{m}$ is a maximal ideal of $R$, then $M_{\mathfrak{m}}$ can be naturally defined as the localization of $M$ at $\mathfrak{m}\ev$ (see \cite[\S5.1]{westra2009superrings} for the details about localization). Thus, according to \cite[Lemma 5.1.21]{westra2009superrings}, the morphism $\phi:M\to N$ is surjective if and only if $M_\m\to N_\m$ is for all maximal ideal $\m$ of $R$. 

Note that $N_\m$ is a finitely generated $R_\m$-supermodule. Thus, $M_\m\to N_\m$ is surjective if and only if $M_\m/\m M_\m\to N_\m/\m N_\m$ is, by \cite[II, \S3, n°2, Corollary 1 to Proposition 4]{bourbaki1998commutative}. The claim then follows from the isomorphisms $M/\m M\simeq M_\m/\m M_\m$ and $/\m N\simeq N_\m/\m N_\m$ that can be proved similar to those in  \cite[II, \S3, n°3, Proposition 9]{bourbaki1998commutative}.

Since $A/\m$ is a field, in view of the claim, i) follows directly as \cite[lemma in p.8]{takeuchi1974tangent}. Furthermore, ii) follows from i) as in \cite[Proposition 1.2.3]{takeuchi1974tangent}.
\end{proof}
 As a corollary of \Cref{lemma.equiv.surj.epim}~ii), we deduce the following proposition.

\begin{proposition}
    Let $\f$ be a map of formal superschemes. The following conditions are equivalent. 
        \begin{enumerate}
            \item[\rm i)] $\f$ is a monomorphism.
            \item[\rm ii)] $\f$ is a closed immersion.
            \item[\rm iii)] $\O_*(\f)$ is injective.  
        \end{enumerate}  
\end{proposition} 

We conclude the present subsection with the following proposition. 

\begin{proposition}\label{Prop.4.6}
    Let $\f:\bX\rightarrow\bY$ be a map of formal superschemes.  The following conditions are equivalent.  
        \begin{enumerate}
            \item[\rm i)] $\f$ is an open immersion; that is, it is a monomorphism and the image-functor is open in $\bY$.
            \item[\rm ii)]  There is a map of formal superschemes $\mathbf{g}:\bW\rightarrow\bY$ such that  $(\f,\mathbf{g}):\bX\coprod \bW\simeq\bY$. 
            \item[\rm iii)] There is a super-coalgebras map $u:\A\rightarrow \O_*(\bY)$ such that $(\O_*(\f),u):\O_*(\bX)\oplus \A\simeq \O_*(\bY)$.      
        \end{enumerate}    
\end{proposition}
\begin{proof} As in \cite[Proposition  1.2 (b)]{takeuchi1977formal}, we will reduce to the finite case.  
The equivalence between ii) and iii) follows from \Cref{3.16}. The implication i) from ii) can be deduced from the construction of the coproduct.  For ii) implies i), we may assume $\bX\subseteq \bY$ is an open subfuntor. Assume that  $\bY=\lim\bY_i$, where the $\bY_i$ are finite $\K$-superschemes. The statement can be reduced to the case when $\bY$ is finite.  Indeed, set $\bX_i$ to be  the preimage  of $\bX$ via the inclusion maps $\bY_i\rightarrow \bY$. Then the $\bX_i$'s form a direct system whose limit is $\bX$. Note that $\bX$ is both open and closed.  Thus, the coordinate ring $\O(\bX)$ is Artinian, by \cite[Corollary 3.4.10]{westra2009superrings}. It follows from \cite[Corollary 5.2.4]{westra2009superrings} that for each $i$, $\O(\bY_i)\simeq \O(\bX_i)\times \O(\bW_i)$,  where $\bW_i$ is the complement of $\bY_i$. This conclude the proof. 
\end{proof}

Immersions of formal superschemes are closed immersions, by \Cref{Prop.4.6}. Therefore, we may refer to closed subfunctors of a formal superscheme $\bX$ as \textit{subsuperschemes } of $\bX$. Note that these are in one-to-one correspondence with sub-super-coalgebra of $\O_*(\bX)$.

\subsection{Fibers}

\begin{definition}
    The \textit{fiber} of a map of formal superschemes $\f:\bX\rightarrow\bY$ over $y\in|\bY|$ is the fiber product $\bX\times_\bY\lbrace y\rbrace$ which is denoted by $\f^{-1}(y)$. 
\end{definition}

\begin{proposition}\label{inm.punt}
    A map of formal superschemes $\f:\bX\rightarrow\bY$ is an immersion if and only if for each $y\in|\bY|$, the map 
    $\f:\f^{-1}(y)\rightarrow \lbrace y\rbrace$ is. 
\end{proposition}

\begin{proof}
    The ``only if'' part is direct. For the ``if part'', we proceed as in  \cite[Proposition 1.3]{takeuchi1977formal} to reduce the requirements to prove that $|\f|:|\bX|\to|\bY|$ is injective. By \Cref{f_eq_f_red}, this is equivalent to the injectivity of $|\f_{\mathrm{bos}}|:|\bX_{\mathrm{bos}}|\to|\bY_{\mathrm{bos}}|$, which follows from the proof of \cite[Proposition 1.3]{takeuchi1977formal}.
\end{proof}

\begin{corollary}
    Consider a cartesian square of formal superschemes  \[
\xymatrix{\bX\ar[d]\ar[rr]^{\f}&& \bY\ar[d]^{\mathbf{g}}\\\bW \ar[rr]_{\mathbf{h}}&&\Z}
\] 
If $\mathbf{g}$ is surjective, then $\f$ is an immersion if and only if $\mathbf{h}$ is so. 
\end{corollary}

\begin{proof}
By \Cref{inm.punt}, it suffices to prove the statement in the case where $\bY = \{y\}$ and $\Z = \{z\}$. Observe that in this situation we have
\[
\O_*(\bW) = \O_*(\bX) \square_{\kappa(y)} \kappa(z).
\]
Hence, it remains to show that a morphism $\O_*(\bX) \to \kappa(y)$ is injective if and only if the induced map
\[
\O_*(\bX) \square_{\kappa(y)} \kappa(z) \to \kappa(z)
\]
is injective. This equivalence is a consequence of the fact that the functor $- \square_{\kappa(y)} \kappa(z)$ is faithfully exact, a property proved in \cite[Lemma 1.7]{zubkov2009affine}.
\end{proof}

\subsection{Flat morphisms} 

An $\O_*(\bX)$-super-comodule is called an $\bX$-\textit{supermodule}.

Let $\mathsf{SMod}_\bX$ be the (abelian) category of $\bX$-supermodules. Let $\f:\bX\to\bY$ be a map of formal superschemes. Thus, we get a forgetful functor $\f_*:\mathsf{SMod}_\bX\to\mathsf{SMod}_\bY$ induced by $\O_*(\f)$. It has a right adjoint given by (\cite[\S2]{takeuchi1977formal})

\[
\function{\f^*}{\mathsf{SMod}_\bY}{\mathsf{SMod}_\bX}{\M}{\M\square_{\O_*(\bY)}\O_*(\bX).}
\]
For any $\M\in\mathsf{SMod}_\bX$ and $x\in|\bX|$, we write 

\[
\M_x=\M\square_{\O_*(\bX)}\O_x\quad \text{and} \quad \M(x)=\M\square_{\O_*(\bX)}\kappa(\bX).
\]

\begin{definition}
    A map $\f:\bX\to\bY$ of formal superschemes  is called (\textit{faithfully}) \textit{flat} if $\f^*$ is (faithfully) exact. Let $x\in|\bX|$ and $y=\f(x)$. We say that $\f$ is \textit{flat} at $x$ if the induced map $\f:\bX_x\to\bY_y$ is.
\end{definition}

\begin{remark}
     Note that $\f$ is flat if and only if  $-\square_{\O_*(\bY)} \O_*(\bX)$ is exact,  or, equivalently,  $-\square_{\O_*(\bY)} \O_*(\bX)$ takes epimorphisms to epimorphisms. 
\end{remark}

In order to obtain useful characterizations of the concepts of flat and faithfully flat, it will be necessary to use the following lemma.

\begin{lemma}\label{super.tak.a.2.2}
    Let \( \A \) be an irreducible super-cocommutative super-coalgebra. A right \( \A \)-super-comodule \( \M \) is flat if and only if it is free, i.e., \( \M \cong \Ws  \otimes \A \) for some super vector space \( \Ws  \). In particular, non-zero flat \( \A \)-super-comodules are faithfully flat.
\end{lemma}

\begin{proof}
    Let $\A_0$ denote the coradical of $\A$. We claim that $\A_0^*$ is a field. Indeed, since $\A$ is super-cocommutative and irreducible, its coradical $\A_0$ is a simple super-coalgebra. Hence, $\A_0^*$ is finite-dimensional, supercommutative, and has no nontrivial proper superideals. In particular, the canonical superideal $\mathfrak{J}_{\A_0^*}$ is zero. Consequently, $\A_0^*$ is a commutative ring whose only proper ideal is the zero ideal, which is therefore maximal. Thus, $\A_0^*$ is a field.
    
    Now, let $\M$ be a flat right $\A$-super-comodule, and consider the cotensor product $\M_0 := \M \square_\A \A_0$. Since $\A_0^*$ is a field, there exists a super vector space $\mathbf{F}$ over $\A_0^*$ such that $\M_0 \cong \mathsf{F} \otimes\A_0$ as $\A_0$-super-comodules (\cite[Propositions~9.2.1 and~9.2.3]{westra2009superrings}). 
    The proof is finished as in \cite[Proposition A.2.2]{takeuchi1977formal}.   
    
\end{proof}

\begin{proposition} Let $\f:\bX\to\bY$ be a map of formal superschemes. 
    \ 
    
    \begin{enumerate}
        \item[\rm i)] $\f$ is flat if and only if it is at every $x\in|\bX|$.
        \item[\rm ii)] The following conditions are equivalent. 
        
        \begin{enumerate}
            \item[\rm (a)] $\f$ is faithfully flat,
            \item[\rm (b)] $\f$ is flat and surjective.
            \item[\rm (c)] $\f$ is flat and strictly surjective.
        \end{enumerate} 
    \end{enumerate}
\end{proposition}

\begin{proof}Let $\A$ be a super-cocommutative super-coalgebra, and let $\{\A_i\}$ denote the set of its irreducible components. Since $$\A = \bigoplus_i \A_i,$$ it follows that for any right $\A$-super-comodule $\M$, we have 

$$\M = \bigoplus_i (\M\square_\A \A_i).$$

\noindent Consequently, such a $\M$ is (faithfully) flat if and only if each of the right $\A_i$-super-comodules $\M \square_\A \A_i$ is (see \cite[p.1527]{takeuchi1977formal}). This, together with  \Cref{super.tak.a.2.2}, implies that the proof of the proposition can be carried out in the same manner as that of \cite[Proposition 2.1]{takeuchi1977formal}.
\end{proof}

\begin{proposition}
    A map $\f:\bX\to\bY$ of formal superschemes is (faithfully) flat if and only if the projection map $\pi_2:\bX\times_\bY \bW\to\bW$ is, for each finite subsuperscheme of $\bY$.
\end{proposition}

\begin{proof}
    Let $\M$ be a right $\A$-super-comodule and $\N$ a left super-comodule over $\A$, with $\N$ finite-dimensional. Then, $\N^*$ has (a uniquely defined) structure of right $\A$-supermodule (\cite[Lemma 1.7]{zubkov2009affine}). Let $\hom_\A(\N^*, \M)$ denote the set of morphisms $\N^*\to \M$ as right $\A$-super-comodules. Then, as in the proof of \cite[Proposition A.2.1]{takeuchi1977formal}, one sees that the functor $\hom_\A(-, \M)$ is (faithfully) exact if and only if its  restriction on the subcategory of finite-dimensional left or right $\A$-super-comodules is. 
\end{proof}

The following two propositions follow straightforward. 

\begin{proposition}
    Let $\xymatrix{\bX\ar[r]^{\f}&\bY\ar[r]^{\mathbf{g}}&\bW}$ be a diagram of formal superschemes and flat maps. Let $x\in|\bX|$,  $y=\f(x)$ and $z=\mathbf{g}(y)$.

    \begin{enumerate}
        \item[\rm i)] If $\f$ is flat at $x$ and $\mathbf{g}$ is at $y$, then $\mathbf{g}\circ \f$ is at $z$.
        \item[\rm ii)] If $\f$ and $\mathbf{g}\circ \f$ are flat at $x$, then $\mathbf{g}$ is at $y$. 
    \end{enumerate}
\end{proposition}

\begin{proposition}
    Let 

    \[
    \xymatrix{\bX\ar[r]^{\f}&\bY\\\bX'\ar[u]^{\mathbf{g}'}\ar[r]_{\f'}&\bY'\ar[u]_{\mathbf{g}}}
    \]
    be a cartesian square of formal superschemes. Let $x'\in|\bX'|$, $x=\mathbf{g}'(x')$ and $y'=\f'(x')$. 

    \begin{enumerate}
        \item[\rm i)] If $\f$ is flat at $x$, then $\f'$ is at $x'$. 
        \item[\rm ii)] If $\f$ is (faithfully) flat then $\f'$ is.
        \item[\rm iii)] If $\f'$ is flat at $x'$ and $\mathbf{g}$ at $y'$, then $\f$ is at $x$.  
    \end{enumerate}
\end{proposition}
 
\begin{lemma}\label{Fact.1.1}
    Let $F\mid \K$ be a field extension, $\M$ be an $\bX$-supermodule and  $z\in|\bX\otimes F|$ lying over $x\in|\bX|$. If $(F\otimes \M)_z=0$, then $\M_x=0$, where $F\otimes\M$ is seen as an $\bX\otimes F$-supermodule.
\end{lemma}

\begin{proof}
The proof can be copied from \cite[Lemma 2.2]{takeuchi1977formal}.
\end{proof}
\begin{proposition}
    Let $\f:\bX\to\bY$ be a map of formal superschemes and $F\mid \K$ a field extension. Let $z\in|\bX\otimes F|$ lie over $x\in|\bX|$. 

    \begin{enumerate}
        \item[\rm i)] $\f$ is flat at $x$ if and only if $\f\otimes F$ is at $z$.
        \item[\rm ii)] $\f$ is (faithfully) flat if and only if $\f\otimes F$ is. 
    \end{enumerate}
\end{proposition}

\begin{proof}
    Suppose that the map $\f\otimes F$ is flat at $z$. Let $\M \to \N$ be a surjective map of $\bY$-supermodules, and $\mathbf{C}$ the cokernel of $\f^*(\M) \to \f^*(\N)$. Observe that the sequence $(F \otimes \f^*(\M))_{z} \to (F \otimes \f^*(\N))_{z} \to 0$ is exact by the assumption, or equivalently $(F \otimes \mathbf{C})_{z} =(F\otimes \mathbf{C})\square_{\O_*(\bX\otimes F)}\O_z= 0.$  By \Cref{Fact.1.1}, $\mathbf{C}_x = 0$, and $\f$ is flat at $x$. The rest is straightforward.
\end{proof}

\begin{theorem}[Faithfully Flat Descent] If $\f:\bX\rightarrow\bY$ is a faithful flat map of formal superschemes, then the diagram \[
\xymatrix{\bX\times_\bY \bX\ar@<-1ex>[rr]_{\pi_2}\ar@<+1ex>[rr]^{\pi_1}&& \bX\ar[rr]^{\f} &&\bY}
\] is exact in the category of formal superschemes.  
\end{theorem} 

\begin{proof} We proceed as in \cite[Theorem 2.4]{takeuchi1977formal}. Let $M$ be a right $\B$-super-comodule and $\psi:\A \to \B$ a map of super-coalgebras. For each $n\geq0$, let 

\[
d_i^n = \underbrace{Id_M\, \Box \cdots \Box Id_M\,}_{i+1} \Box \psi \Box \underbrace{Id_M\, \Box \cdots \Box Id_M\,}_{n-i} : M \Box_\B \underbrace{\A \Box_\B \cdots \Box_\B \A}_{n+1} \to M \Box_\B \underbrace{\A \Box_\B \cdots \Box_\B \A}_{n}.
\]
and consider
\[
\partial_n = \sum_{i=0}^n (-1)^i d_n^i.
\]
More explicitly, 
 \begin{align*}
     \partial_0(m\square_\B \A)&=m\cdot\psi(a)\\
     \partial_1(m\square_\B a_1\square_\B a_2)&=(d_0^1-d_1^1)(m\square_\B a_1\square_\B a_2)=m \square_\B \psi(a_1) \square_\B a_2 
- m \square_\B a_1 \square_\B \psi(a_2)\\
&\ \vdots\ \\
\partial_n(m \square_\B a_1 \square_\B \cdots \square_\B a_n)
&= \sum_{i=0}^n (-1)^i \, m \square_\B a_1 \square_\B \cdots \square_\B a_{i-1} \square_\B \psi(a_i) \square_\B a_{i+1} \square_\B \cdots \square_\B a_n.
 \end{align*}
Then, we have a complex $$0 \leftarrow   M \xleftarrow{\partial^0} M \Box_\B \A \xleftarrow{\partial^1} M \Box_\B \A \Box_\B \A \xleftarrow{\partial^2} \cdots$$

This complex is acyclic after applying the functor $- \Box_\B \A$ (\cite[p.1494]{takeuchi1977formal}). Hence, it is acyclic if $\psi$ is faithfully flat. If $M = \B$, applying the functor $\sp^*$, we obtain the desired exact sequence.
\end{proof}

\section{Locally algebraic formal superschemes}\label{sec.loc.alg}

Thorough this section, and unless we explicitly state it otherwise, $F\mid \K$ is a field extension and formal superschemes $\bX, \bY, \ldots$ are considered over the field $\K$.  

\begin{definition}
\ 

    \begin{enumerate}
        \item[i)]  We say that $\bX$ is \textit{algebraic at} $x\in |\bX|$ if the super-coalgebra $\O_x$ is of finite type.
        \item[ii)] We say that $\bX$ is \textit{locally algebraic} if $\bX$ is algebraic in every point, or equivalently, $\O_*(\bX)$ is locally finite.
    \end{enumerate}    
\end{definition}

\begin{proposition}\label{prop.prod.of.loca.alg}
    The product of locally algebraic superschemes is locally algebraic. 
\end{proposition}

\begin{proof}
    Let $\bX$ and $\bY$ be formal superschemes algebraic at $x$ and $y$, respectively. Let $z\in\bX\times\bY$ lying over $(x, y)$, then $\O_x$ and $\O_y$ are of finite type, by definition. Thus, by \Cref{tensor.fin.type}, $\O_z=\O_x\otimes\O_y$ is of finite type and $\bX\times\bY$ is algebraic at $z$. 
\end{proof}

\begin{proposition}\label{T.3.1}
    If $z\in|\bX\otimes F|$ lies over $x\in|\bX|$, then $\bX$ is algebraic at $x$ if and only if $\bX\otimes F$ is algebraic at $z$. In particular, $\bX$ is locally algebraic if and only if $\bX\otimes F$ is.
\end{proposition}

\begin{proof} Use \Cref{Fact.1.1} and \Cref{prop.equiv.fint.type} to copy the proof from \cite[Proposition 3.1]{takeuchi1977formal}.    
\end{proof}

\subsection{LOFP and OFP maps}

\begin{definition}
    Let $\mathbf{f}:\bX\to\bY$ be a map of formal superschemes. 
    
    \begin{enumerate}
        \item[i)] $\f$ is \textit{locally of finite presentation} (or LOFP, for brief) if there is a locally algebraic formal superscheme $\bW$ and an immersion $i$ such that  the following diagram commutes 
    \[
    \xymatrix{\bX\ar[rr]^{i}\ar[drr]_{\mathbf{f}}&&\bW\times\bY\ar[d]^{\pi_2}\\&&\bY}
    \]
    \item[ii)]  $\mathbf{f}$ is \textit{of finite presentation} (or briefly, OFP) at $x\in|\bX|$ if $\mathbf{f}:\bX_x\to\bY_{\mathbf{f}(x)}$ is LOFP.
    \end{enumerate}
    
\end{definition}

\begin{proposition}
    A map of formal superschemes is LOFP if and only if it is OFP everywhere.
\end{proposition}

\begin{proof}  If $\f:\bX\to \bY$ is LOFP, then there is some locally algebraic formal superscheme $\bW$ and an immersion $i:\bX\to\bW\times\bY$ such that 
\[
\f=\pi_2\circ i:\bX\to\bW\times\bY\to\bY.
\]
If $x\in|\bX|$, then we get an immersion $i|_{\bX_x}:\bX_x\to\bW_{i(x)}\times\bY_{\f(x)}$ with 
\[
\f=\pi_2\circ i|_{\bX_x}:\bX_x\to\bW_{i(x)}\times\bY_{\f(x)}\to\bY_{\f(x)}.
\]
Hence $\f$ is OFP at $x$, for all $x\in|\bX|$.

Conversely, suppose that $\f$ is OFP everywhere. Then, for each $x\in|\bX|$, there is some locally algebraic formal superscheme $\bW(x)$ and an immersion $i(x)$ such that 

\[
\f=\pi_2\circ i(x):\bX_x\to\bW(x)\times \bY_{\f(x)}\to\bY_{\f(x)}.
\]
We thus have an immersion 

$$i:\bX=\coprod_{x\in|\bX|}\bX_x\to\bW\times\bY=\left(\coprod_{x\in|\bX|}\bW(x)\right)\times\bY$$ such that $\f=\pi_2\circ i:\bX\to\bW\times\bY\to\bY$, and $\f$ is then LOFP. 
\end{proof}

\begin{proposition}\label{Prop.equiv.LOPF.and.OFP}
    \ 

    \begin{enumerate}

        \item[\rm i)] If $\mathbf{f}:\bX\to\bY$ is LOFP, then $\pi_\bW:\bW\times_\bY\bX\to\bW$ is, for each $\bW\to\bY$.
        \item[\rm ii)] If $\mathbf{f}:\bX\to\bY$ and $\mathbf{g}:\bY\to\bW$ are LOFP, then $\mathbf{g}\circ\mathbf{f}:\bX\to\bY$ is.
        \item[\rm iii)] If $\mathbf{g}\circ\mathbf{f}:\bX\to\bY\to\bW$ is LOFP, then $\mathbf{f}:\bX\to\bY$ is.
    \end{enumerate}
\end{proposition}

\newpage

\begin{proof}
\ 

\begin{enumerate}
    \item[i)] Let $\Ub $ be a locally algebraic formal superscheme and $i_\bX:\bX\to\Ub \times\bY$ proving that $\f$ is LOFP. We thus get an immersion $i:\bW\times_\bY \bX\to\Ub \times\bW$, where, for each $R$, $(w, x)\in\bW(R)\times_{\bY(R)}\bX(R)$ is mapped to $(\pi_{\Ub }(R
    )(i_\bX(R)(x)), w)$. Then,  $\pi_\bW=\pi_2\circ i$, and $\pi_\bW$ is LOFP. 
   \item[ii)] Let $\Ub $ be a locally algebraic formal superscheme and $i_\bX:\bX\to\Ub \times\bY$ an immersion proving that $\f$ is LOFP. Let $\V$ and $i_\bY:\bY\to\V\times\bW$ the corresponding data for $\g$. We thus have a locally algebraic formal superscheme $\Ub \times\V$ (\Cref{prop.prod.of.loca.alg}) and an immersion $i:\bX\to(\Ub \times\V)\times\bY $, where, for each $R$, if $i_\bX(R):x\mapsto (x_1, x_2)$ and $i_\bY(R):x_2\mapsto(y_1, y_2)$, then $i(R):x\mapsto((x_1, y_1), (\g\circ\f)(R)(x_2))$. Then,  $\g\circ  
    \f=\pi_2\circ i$, and $\g\circ\f$ is LOFP. 
    \item[iii)] Here we proceed as in \cite[p.1496]{takeuchi1977formal}. Consider the following diagram composed of cartesian squares 
    \[
    \xymatrix{\bX\ar[rr]^{\f}&&\bY\ar[rr]^{\mathbf{g}}&&\bW\\&&&&\\\bX\times_\bW\bY\ar[uu]^{\pi_1}\ar[rr]^{\f\times_\bW\bY}&&\bY\times_\bW\bY\ar[uu]^{\pi_1}\ar[rr]^{\pi_2}&&\bY\ar[uu]_{\mathbf{g}}\\&&&&\\
    \bX\ar[uu]^{( Id_\bX, \f)}\ar[rr]_{\f}&&\bY\ar[uu]_{( Id_\bY,  Id_\bY)}&&
}    \]
Observe that if $\mathbf{g}\circ\f$ is LOFP, then $\pi_2\circ(\f\times_\bW\bY)$. Also, note that $( Id_\bX, \f)$ is LOFP. Therefore $$\f=\pi_2\circ(\f\times_\bW\bY)\circ( Id_\bX, \f)$$ is LOFP, by ii).
\end{enumerate}
\end{proof}

\begin{proposition}\label{P.5.7}
    Let $\mathbf{f}:\bX\to\bY$ be a map of formal superschemes. Then $\f$ is OFP at $x\in|\bX|$ if and only if  $\mathbf{f}^{-1}(y)$ is algebraic at $x$, where $y=\mathbf{f}(x)$. In particular, $\mathbf{f}$ is LOFP if and only if $\mathbf{f}^{-1}(y)$ is locally algebraic for all $y\in|\bY|$. 
\end{proposition}

\begin{proof} Use \Cref{lemma.emd.fn.tp} and \Cref{inm.punt} to proceed as in \cite[Proposition 3.2]{takeuchi1977formal}.    
\end{proof}

\begin{corollary}
    Let $\f:\bX\to\bY$ be a map of formal superschemes. The following conditions are equivalent. 

    \begin{enumerate}
        \item[\rm i)] $\f$ is LOFP. 
        \item[\rm ii)] For each locally algebraic subsuperscheme $\bW$ of $\bY$, $\f^{-1}
        (\bW)$ is locally algebraic. 
        \item[\rm iii)] There is a locally algebraic subsuperscheme $\bW$ of $\bY$, where $|\bW|=|\bY|$ and $\f^{-1}(\bW)$ is locally algebraic.  
    \end{enumerate}
\end{corollary}

\begin{proof}
    The equivalence i) $\Leftrightarrow$ ii) is direct from \Cref{P.5.7}. For the equivalence with iii), it suffices to take $\bW=\coprod\limits_{y\in|\bY|}\{y\}.$
\end{proof}

\begin{proposition}
    Let $\mathbf{f}:\bX\to\bY$ be a map of formal superschemes, and $z\in|\bX\otimes F|$ lying over $x\in|\bX|$. Then $\mathbf{f}$ is OFP at $x$ if and only if $\mathbf{f}\otimes F$ is at $z$.
\end{proposition}

\begin{proof}
    The same as in \cite[Proposition 3.3]{takeuchi1977formal}.
\end{proof} 

\begin{definition}
    Let $V$ be a $\bX$-supermodule, where $\bX$ is a formal superscheme. We say that $V$ is \textit{finitely cogenerated} if it is isomorphic to some $\bX$-subsupermodule of $\O_*(\bX)^n$ for some integer $n$.
\end{definition}

\begin{definition}
    A map $\mathbf{f}:\bX\to\bY$ of formal superschemes is 

    \begin{enumerate}
        \item[i)] \textit{finite} if for each finite subsuperscheme $\bW$ of $\bY$, $\mathbf{f}^{-1}(y)$ is finite, 
        \item[ii)] \textit{finite bounded} if the $\bY$-supermodule $\mathbf{f}_*(\O_*(\bX))$ is finitely cogenerated.
    \end{enumerate}
\end{definition}

\begin{proposition}
    A map of formal superschemes $\mathbf{f}:\bX\to\bY$ is finite if and only if $\mathbf{f}^{-1}(y)$ are finite subsuperschemes of $\bX$ for all $y\in|\bY|$. In particular, finite bounded maps are finite. 
\end{proposition}

\begin{proof}
This follows in the same way as in \cite[Proposition 3.4]{takeuchi1977formal}.
\end{proof}

\subsection{Local superdimension} Let $\bX$ be a locally algebraic formal superscheme. Then, $\O_x$ is of finite type for each $x\in|\bX|$. By  \Cref{prop.equiv.fint.type} and \cite[I, \S3, Corollary 5.10]{demazure1980introduction}, we find that $\O_x^*$ is Noetherian with $ \kdim(\O_x^*)\ev<\infty$. Furthermore, 

\[
\kdim\,\O_*(\bX)\ev=\sup_{x\in|\bX|}\kdim\,(\O_x)\ev
\]
Below we assume that this number is finite. 

\begin{definition}
    Let $\bX$ be a locally algebraic formal superscheme and $x\in|\bX|$. We put 

    \[
    \sdim_x\,\bX:=\ksdim\,\O_x\quad \text{ and }\quad\sdim\ \bX=\ksdim\,\O_*(\bX).
    \]
    We call these vectors the \textit{local superdimension} of $\bX$ at $x$ and the \textit{superdimension} of $\bX$, respectively.
\end{definition}

Below we use the notation $\sdim_{x, \overline{0}}\,\bX$ and $\sdim_{x, \overline{1}}\,\bX$ to denote the even and odd components of the local superdimension, respectively. For the superdimension of $\bX$, we keep the same convention. 

\begin{proposition}
    Let $\bX$ be a locally algebraic formal superscheme. Then for any $x\in|\bX|$ we have 
    \[\sdim_{x,\overline{0}}\,\bX+\sdim_{x,\overline{1}}\,\bX<\infty.
    \] 
\end{proposition}

\begin{proof} Let $R=\O_x^*$. Then, $\sdim_x\,\bX=\ksdim\,R$. Note that the superring $R$ is local Noetherian, so $R\ev$ is also local Noetherian. Hence, by \cite[I, \S 3,  Corollary 5.10]{demazure1980introduction} the even component of the superdimension is finite. The odd component is finite by Noetherianity. 
\end{proof}
 
To state our next theorem, we need to define regular superrings. Let $(R, \m)$ be a local Noetherian superring. If $\ksdim\,R=\sdim_{R/\m}\,\m/\m^2$, we say that $R$ is \textit{regular}. For a non-local $R$, we say that it is \textit{regular} if all the $R_\m$ are regular, with $\m$ varying on the maximal ideals of $R$. Refer to \cite[\S5]{masuoka2020notion} for a wider theory of regular superrings. 

\begin{theorem}\label{ine.sdim}
    Let $\mathbf{f}:\bX\to\bY$ be a map of locally algebraic formal superschemes, $x\in|\bX|$ and $y=\mathbf{f}(x)$. Then 

    \[
    \sdim_{x, \overline{0}}\,\bX\leq\sdim_{y, \overline{0}}\,\bY+\sdim_{x, \overline{0}}\,\mathbf{f}^{-1}(y).
    \]
    If $\mathbf{f}$ is flat at $x$, then the equality holds. If $\O_y^*$ is also regular and both $\O_x^*$ and $\O_y^*$ contain a field, then 

    \begin{equation}\label{clubsuit}
        \sdim_{x, \overline{1}}\,\bX\geq\sdim_{y, \overline{1}}\,\bY+\sdim_{x, \overline{1}}\,\mathbf{f}^{-1}(y).
    \end{equation}
\end{theorem}

\begin{proof}
      Observe that the superalgebras $R=\O_y^*$ and $S=\O_x^*$ are local Noetherian and $\f$ induces a local morphism $\phi:R\to S$. The maximal superideal of $R$ is $\kappa(y)^\perp$. The superideal $S\m$ of $S$ is $D^\perp$, with $D=\O_*(\f^{-1}(y)_x)$. Hence $\ksdim\,D=\sdim
      _x\,\f^{-1}(y)$ and therefore \cite[I, \S\,3, Corollary 5.12]{demazure1980introduction} implies that 
      
      \[\sdim_{x, \overline{0}}\,\bX=\ksdim\ev\,S\leq\ksdim\ev\,R+\ksdim\ev(S/S\m)=\sdim_{y, \overline{0}}\,\bY+\sdim_{x, \overline{0}}\,\f^{-1}(y).\]

      Now, $R^\circ=\O_y$ and $S^\circ=\O_x$ so, if $\f$ is flat, then the equality occurs above by \cite[Corollary 1.7.9]{takeuchi1974tangent}. Suppose further that $R$ is regular and both $R$ and $S$ contain a field. Thus, the morphism $\phi:R\to S$ is flat and by \cite[Theorem 7.2]{zubkov2022dimension}, we obtain that 
      \[
    \sdim_{x, \overline{1}}\,\bX\geq\sdim_{y, \overline{1}}\,\bY+\sdim_{x, \overline{1}}\,\mathbf{f}^{-1}(y).
    \]      
\end{proof}

\begin{question}
    When does the equality occurs in \eqref{clubsuit}?
\end{question} 

\begin{theorem} Let $\K$ be algebraically closed. Consider $\bX$ and $\bY$ be formal superschemes algebraic at  $x$ and $y$, respectively. Suppose that $z\in\bX\times\bY$ lies over $(x, y)$, $\O_y^*$ is regular and both $\O_x^*$ and $\O_y^*$ contain a field. Then 

    \[
     \sdim_{z, \overline{0}}\,(\bX\times\bY)=\sdim_{x, \overline{0}}\,\bX+\sdim_{z, \overline{0}}\,\bY \quad \text{and}\quad 
     \sdim_{z, \overline{1}}\,(\bX\times\bY)\geq\sdim_{x, \overline{1}}\,\bX+\sdim_{z, \overline{1}}\,\bY.
    \] 
\end{theorem}

\begin{proof}
    Under the assumptions, $\pi_1:\bX\times\bY\to\bX$ is flat. Thus, by \Cref{ine.sdim} we have  

    \[
    \sdim_{z, \overline{0}}\,(\bX\times\bY)=\sdim_{x, \overline{0}}\,\bX+\sdim_{z, \overline{0}}\,\pi_1^{-1}(x) \quad \text{and}\quad 
     \sdim_{z, \overline{1}}\,(\bX\times\bY)\geq\sdim_{x, \overline{1}}\,\bX+\sdim_{z, \overline{1}}\,\pi_1^{-1}(x).
    \]
    The proof then follows from the isomorphism $\pi_1^{-1}(x)=\{x\}\times\bY\simeq\bY$.
\end{proof}

\begin{example} Consider the \textit{parity shift functor} $$\Pi:\mathsf{SVec}_\K\to\mathsf{SVec}_\K, V\mapsto \Pi\,V,$$ where, for each $i\in\Z_2$, $(\Pi V)_i=V_{i+\overline{1}}$. On morphisms, $\Pi\phi=\phi$, for any morphism $\phi:V\to W$ between vector superspaces. If $v\in V^h$, then $v\in(\Pi V)^h$ with opposite parity, so we denote it for $\Pi v$. Now consider $R$ to be as in \cite[Example 7.6]{zubkov2022dimension} and $y=\Pi1$. Thus, $S=\K[y]$ is a $\K$-superalgebra with $S\ev=\K$ and $\mathfrak{J}_S=S\od=\K y$. It is proved in \cite[Example 7.6]{zubkov2022dimension} that the embedding $\K[y]\to R$ is flat and 

\[
\ksdim\od\,R>\ksdim\od\,\K[y]+\ksdim\od\,R/R\mathfrak{M},
\]
where $\mathfrak{M}=\K y$ is the maximal ideal of $S$. As in the proof of \Cref{ine.sdim}, wee see that for the formal superschemes $\bX=\spf\,\K[y]$ and $\spf\,R$, the map $\mathbf{i}:\bX\to\bY$ induced by $i$ is flat, and for each $x\in|\bX|$, we have 

\[
\sdim_{x, \overline{0}}\,\bX=\sdim_{y, \overline{0}}\,\bY+\sdim_{x, \overline{0}}\,\mathbf{f}^{-1}(y)\text{ and }\sdim_{x, \overline{1}}\,\bX>\sdim_{y, \overline{1}}\,\bY+\sdim_{x, \overline{1}}\,\mathbf{f}^{-1}(y). 
\]
\qed 
\end{example}

In the following example, we use the notation and terminology from \cite{HOSHI202028} and \cite{takahashi2024quotients}.

\begin{example} 
Let $\mathbf{H}$ be an associative $\K$-superalgebra that is also a supercoalgebra. Suppose that the coproduct $\Delta: \mathbf{H} \to \mathbf{H} \otimes \mathbf{H}$ and the counit $\epsilon: \mathbf{H} \to \K$ are morphisms of superalgebras, that is,
\[
\Delta(xy) = \Delta(x)\Delta(y), \qquad \epsilon(xy) = \epsilon(x)\epsilon(y).
\]
If, in addition, $\mathbf{H}$ is endowed with an even linear map $S: \mathbf{H} \to \mathbf{H}$ (called the \emph{antipode}) such that for all $x \in \mathbf{H}$, writing $\Delta(x) = \sum x_{(1)} \otimes x_{(2)}$, we have
\[
\epsilon(x) = S(x_{(1)}) x_{(2)} = x_{(1)} S(x_{(2)}),
\]
then $\mathbf{H}$ is called a \emph{Hopf superalgebra} (see \cite{westra2009superrings} for further details).

Now, let $\mathbf{J}$ be a Hopf superalgebra and $\mathbf{I}$ a Hopf subsuperalgebra of $\mathbf{J}$. Denote by $J:=\Delta_\mathbf{J}^{-1}(\mathbf{J}\ev\otimes\mathbf{J}\ev)$ (resp.~$I:=\Delta_\mathbf{I}^{-1}(\mathbf{I}\ev\otimes\mathbf{I}\ev)$) the largest purely even Hopf subsuperalgebra of $\mathbf{J}$ (resp.~$\mathbf{I}$). Let
\[
P(\mathbf{J}) = \{ v \in \mathbf{J} \mid \Delta(v) = 1 \otimes v + v \otimes 1 \}
\]
be the space of \textit{primitive elements} of $\mathbf{J}$, and let $\Vs _{\mathbf{J}}$ (resp.~$\Vs _{\mathbf{I}}$) be its odd component relative to $\mathbf{J}$ (resp.~$\mathbf{I}$). These are purely odd right (and left) supermodules over $J$ and $I$, respectively. We can regard $I$ as a Hopf subalgebra of $J$, and $\Vs _{\mathbf{I}}$ as an $I$-submodule of $\Vs _{\mathbf{J}}$. Thus, we may define the quotient
\[
\mathsf{Q} = \Vs _{\mathbf{J}} / \Vs _{\mathbf{I}}.
\]
Now define the quotients $\mathbf{J}/\hspace{-1.2mm}/\,\mathbf{I}$ and $J/\hspace{-1.2mm}/I$ as in \cite[\S3]{takahashi2024quotients}.

\begin{enumerate}
    \item[i)] Assume that $\dim_\K\Vs _\mathbf{J}=m<\infty$, and that $J = \mathfrak{Cof}(\Us)$, where $\Us\in\mathsf{SVect}_\K$ is purely even and $\dim_\K \Us = n < \infty$. (If $J$ is \textit{connected}, that is, if $\dim_\K(\cod\,J)=1$, the \textit{smoothness} of $J$ is a sufficient condition for $J=\mathfrak{Cof}(\Us)$. Under this assumption, it follows from \cite[Proposition~3.12]{HOSHI202028} that $J = \mathfrak{Cof}(P(J))$.) Then, \cite[Eq.~(2.11)]{takahashi2024quotients} implies that 
    \[
    \mathbf{J} \simeq J \otimes \wedge(\Vs _{\mathbf{J}}).
    \]
    Therefore, using \cite[Eq.~(3.7)]{HOSHI202028}, we obtain
    \[
    \mathbf{J}^* \simeq (J \otimes \wedge(\Vs _{\mathbf{J}}))^* \simeq \mathfrak{Cof}(J \oplus \Vs _{\mathbf{J}})^* \simeq \K\llbracket T_1, \ldots, T_{n+m}\rrbracket,
    \]
    where $T_1, \ldots, T_n$ are even variables and $T_{n+1}, \ldots, T_{n+m}$ are odd variables. Thus,
    \[
    \ksdim\,\sp^* \mathbf{J} = \kdim \,\mathrm{Sp}^* J \mid \dim_\K \Vs _{\mathbf{J}}.
    \]

    \item[ii)] Suppose that one of the conditions i)–iv) in \cite[Proposition~4.4]{takahashi2024quotients} is satisfied. Suppose further that $\mathbf{J} /\hspace{-1.2mm}/\, \mathbf{I}$ (or equivalently $J/\hspace{-1.2mm}/\, I$) is smooth \cite[Theorem~4.8]{takahashi2024quotients} and that $J/\hspace{-1.2mm}/I$ is connected. Then
    \[
    \mathbf{J} /\hspace{-1.2mm}/\, \mathbf{I}\simeq J\otimes_I\wedge(\mathsf{Q})\simeq J /\hspace{-1.2mm}/\, I\otimes\wedge(\mathsf{Q}).
    \]
    Therefore, we find that
    \begin{align*}
        \ksdim\,\sp^*(\mathbf{J} /\hspace{-1.2mm}/\, \mathbf{I})
        &=\kdim\,\mathrm{Sp}^*(J /\hspace{-1.2mm}/\, I)\mid\dim_\K\Vs _\mathbf{J}-\dim_\K\Vs _\mathbf{I}.
    \end{align*}
    In particular, if $I=\K$, then $ \mathbf{J} /\hspace{-1.2mm}/\, \mathbf{I}\simeq J\otimes\wedge(\mathsf{Q})$, so that 
    
    \[
    \ksdim\,\sp^*(\mathbf{J} /\hspace{-1.2mm}/\, \mathbf{I})=\kdim\,\mathrm{Sp}^*J\mid\dim_\K\Vs _\mathbf{J}-\dim_\K\Vs _\mathbf{I}.
    \]
\end{enumerate}
\end{example}

\section{Further research directions}\label{sec.forthcoming}

Building on the results of this work, we envision several possible directions for further exploration.

A first line of inquiry concerns the notion of \textit{Krull superdimension} for super-cocommutative super-coalgebras. It seems natural to seek an intrinsic definition, not relying on the dual superalgebra, and to identify conditions under which the equality $\ksdim\,\A = \ksdim\,\A^*$ holds for a finite type super-coalgebra $\A$.

Another direction involves the study of \textit{constant} and \textit{étale} (formal) superschemes. Developing a general framework for these objects could help clarify their categorical properties, their functorial behavior, and how they relate to their bosonic and even reductions.

Finally, it would be interesting to adapt the classical notions of \textit{non-ramified}, \textit{étale}, and \textit{smooth} morphisms to the setting of formal superschemes. Understanding how these concepts interact with the corresponding supermodules of differentials may shed light on the geometry of formal superspaces.

\section*{Acknowledgments}

The second-named author gratefully acknowledges the hospitality and partial financial support of the Instituto de Matemática Pura e Aplicada (IMPA, Brazil) during his visit in January–February 2025, when part of this project was carried out. 

%%%%%%%%%%%%%%%%%%%%%%%%%%%%%%%%%%%%%
%%%%%%%%%%%%%%%%%%%%%%%%%%%%%%%%%%%%%
%%%%%%%%% REFERENCES %%%%%%%%%%%%%%%%
%%%%%%%%%%%%%%%%%%%%%%%%%%%%%%%%%%%%%
%%%%%%%%%%%%%%%%%%%%%%%%%%%%%%%%%%%%%

\end{document}